\documentclass[11pt,A4paper]{amsart}
\usepackage{verbatim, amscd, epsfig,amsmath,amsthm,amstext,amssymb,hyperref,tikz, enumerate
}
\usepackage[latin1]{inputenc}
\usepackage{amssymb, booktabs, longtable}
\usepackage[margin=3.8cm]{geometry}

\theoremstyle{plain}
\newtheorem*{theorem*}{Theorem}
\newtheorem{theorem}{Theorem}[section]
\newtheorem{lemma}[theorem]{Lemma}

\theoremstyle{definition}
\newtheorem{definition}[theorem]{Definition}

\newtheorem {proposition}[theorem] {Proposition}

\newcommand{\R}{\mathbb{ R}}
\newcommand{\C}{\mathbb{ C}}
\newcommand{\Z}{\mathbb{ Z}}

\DeclareMathOperator{\im}{im}
\DeclareMathOperator{\Span}{span}
\DeclareMathOperator{\Id}{Id}

\newcommand{\g}{\mathfrak{g}}
\newcommand{\h}{\mathfrak{h}}

\renewcommand{\t}{\mathfrak{t}}
\newcommand{\mf}{\mathfrak}
\newcommand{\bz}{{\bar{z}}}

\newcommand{\Ad}{\operatorname{Ad}}
\newcommand{\ad}{\operatorname{ad}}
\newcommand{\inv}{^{-1}}

\newcommand{\p}{\mathfrak{p}}
\newcommand{\Tor}{\C/\Lambda}
\newcommand{\D}{\nabla_\lambda}
\begin{document}
\title
{Toda Frames, Harmonic maps and extended Dynkin diagrams}
\author{Emma Carberry, Katharine Turner}
\date{\today}
\maketitle
 %\tableofcontents

\begin{abstract}
 We prove that all immersions of a genus one surface into $ G/T $ possessing a Toda frame can be constructed by integrating a pair of commuting vector fields on a finite dimensional Lie algebra. Here $ G $ is any simple real Lie group (not necessarily compact),
% which is preserved by a Coxeter automorphism
$ T $ is a Cartan subgroup and the $k$-symmetric space structure on $ G/T $ is induced from the Coxeter automorphism. We provide necessary and sufficient  conditions for the existence of a Toda frame for a harmonic map into  $G/T$
 %in terms of the coefficients of the Chevalley basis of a cyclic primitive lift.
and describe those $G/T$ to which the theory applies in terms of involutions of  extended Dynkin diagrams.

\end{abstract}

\section{Introduction}

%!TeX root=finite.tex

The last few decades have seen significant progress in the understanding and classification of harmonic maps from surfaces into  compact real Lie groups and symmetric spaces.
An important class of harmonic maps are those of  \emph{finite type}, which are obtained as the solutions to  a pair of  ordinary  differential equations on a finite dimensional loop algebra. This is a far simpler process than attempting to solve the Laplace-Beltrami equation directly, and so motivates us to determine circumstances under which harmonic maps are of finite type. Similarly, when the target manifold is a $ k $-symmetric space, $ k >2 $,
it is natural to restrict our attention to those harmonic maps  which are cyclic primitive and ask when these maps are of finite type. Many papers  (e.g. \cite {Hitchin:90, PS:89, Bobenko:91, FPPS:92, BFPP:93, BPW:95, Burstall:95})
 have addressed these questions when the target Lie group or ($ k $)-symmetric space is {\em compact}.
%}, so that in particular we know that all non-conformal harmonic maps from a genus one surface into a rank one compact symmetric space are of finite type  \cite {BFPP:93}.
We remove the need for this compactness assumption
and in Theorem~\ref {thm:finite} show that all  maps from a genus one surface into a $ k $-symmetric space $ G/T $ possessing a Toda frame are of  finite type, where $ G $ is  any simple real Lie group preserved by a Coxeter automorphism and $ T $ is the corresponding Cartan subgroup. A natural generalisation of the usual 2-dimensional affine Toda field equations provides the integrability condition for the existence of a Toda frame, and so we make contact with classical integrable systems theory.
 To determine the spaces $ G/T $ and the harmonic maps into them to which this theory applies we  address the following two questions, each of independent interest:
\begin {enumerate}
\item {\em When does a map from a surface into $ G/T $ possess a Toda frame?} and
\item {\em When is $ G $ preserved by a Coxeter automorphism?}
\end {enumerate}
The first of these is answered in Theorem~\ref {theorem:Toda}, where it is proven that a map from a surface into $ G/T $ locally has a Toda frame precisely when it is cyclic primitive and a certain function is constant. Cyclic primitive maps are in particular harmonic and play an analogous role for $ k $-symmetric spaces as harmonic maps do for symmetric spaces.
This and our finite-type result are the natural extensions of results obtained in \cite {BPW:95} in the case when $ G $ is compact. The second question does not arise in the compact situation, since  a Coxeter automorphism for a complex simple Lie algebra $\g ^\C $ automatically preserves a compact real form $\g $. % When $\t =\g\cap\t ^\C $ is maximally compact, so in particular when $ G $ is compact.
We characterise when a Coxeter automorphism preserves a real form of a complex simple Lie algebra, which is equivalent to the corresponding real Lie group $ G $ being preserved whenever $ G $ simply connected or adjoint. Given simple roots for  $\g^\C $ spanning a Cartan subalgebra $\t ^\C $, let $\sigma $ be
 the associated Coxeter automorphism and $\Theta $ a Cartan involution with respect to $\g $ that preserves $\t =\g\cap\t ^\C $.
Then  $\sigma $ preserves $\g $ if and only if $\Theta $ defines a permutation of the extended Dynkin diagram, so in particular whenever $\t $ is a maximally compact Cartan subalgebra (Proposition~\ref{prop:Coxeter}).
 In Theorem~\ref{thm:3.2}
we prove that all involutions of the extended Dynkin diagram for a simple complex Lie algebra $\g ^\C $ arise from a Cartan involution for some real form $\g $.

Harmonic maps from surfaces into Lie groups and symmetric spaces arise naturally in many geometric and physical problems. On the geometric side, strong motivation comes from the study of surfaces with particular curvature properties. For example, minimal surfaces are described by conformal harmonic maps and both constant mean curvature and Willmore surfaces are characterised by having harmonic Gauss maps into particular symmetric spaces. From the physics viewpoint, these harmonic maps are interesting because of their relationship with the appropriate Yang-Mills equations and  non-linear sigma-models. Indeed  the harmonic map equations on a Riemann surface are precisely the reduction of the Yang-Mills equations on $\R ^ {2, 2} $ obtained by considering solutions invariant under translation in the directions of negative signature. Classical solutions of sigma-models are given by harmonic maps into (non-compact) as pseudo-Riemannian manifolds.  In  \cite{CT:12} we study an explicit example, namely  harmonic tori in de Sitter spaces $ S ^ {m}_1 $. In particular we apply the theory of this paper to the superconformal such maps with globally defined harmonic sequence to see that they may all be obtained by integrating a pair of commuting vector fields on a finite-dimensional vector space. It follows that all Willmore tori in $ S ^ 3 $ without umbilic points may be obtained in this simple way.

% The fact \cite {Pohlmeyer:76,  Uhlenbeck:89} that a harmonic map from a surface to a Lie group corresponds to a loop of flat connections is the fundamental observation that enables the application of integrable systems techniques to the study of these %maps. The Cartan map $ G/H\rightarrow G $ from a symmetric space to the relevant Lie group is well-known to be a totally geodesic immersion when $ G $ is compact and equipped with a bi-invariant Riemannian metric (which must exist in the compact case). %The composition of a harmonic map with a totally geodesic one is again harmonic, so this enables harmonic maps into symmetric spaces to be studied using the same tools as those into Lie groups, and in particular in terms of a loop of flat connections. We %show \ref{thm2.1} % Clinton: could you please add a reference to theorem 2.1 on page 4
% that when $ G $ has merely a bi-invariant pseudo-metric that the Cartan map is again a totally geodesic immersion. In particular all reductive Lie groups possess a bi-invariant pseudo-metric. This enables us to study harmonic maps into  $ G/H $ using %integrable systems methods regardless of whether $ G $ is compact.
%

The structure of this paper is as follows. In section~\ref {symmetric} we give the general theory for harmonic maps of surfaces into symmetric spaces and for primitive maps into $ k $-symmetric spaces when the relevant Lie group $ G $ is equipped with a bi-invariant pseudo-metric. % We explain how these maps are each described by a loop of flat connections, which is the fundamental observation that enables the application of integrable systems techniques.
% This is a natural generalisation of the special case where the target is compact.
The question of when a Coxeter automorphism preserves the real form of the complex simple Lie algebra is addressed in section ~\ref{dynkin} in terms of Cartan involutions and extended Dynkin diagrams.
% In section~\ref {primitive} we derive the conditions under which a harmonic map of the complex plane into $ S ^ {2n}_1 $ has a primitive lift into the $ 2n $-symmetric space defined by the Coxeter automorphism on $ SO (2n, 1) $.
Section~\ref {Toda} contains the relationship with the affine Toda field equations and the finite type result is proven in section~\ref {finite}.
% The final section contains the applications to Willmore surfaces in $ S ^ 3 $. % Clinton: please add 2 the various source files found in finitenew tags so that the above commands yield section numbers.

%Harmonic maps into compact 2-symmetric spaces correposond to loops of flat connections. In section 1 we establish the same correspondence when the Lie group is not compact but is equipped with a bi-invariant pseudo-metric.  From this we can observe that the projection of certain primitive maps are harmonic. In this %section we also define harmonic maps of finite type in terms of existence of an adapted polynomial Killing field.
%
%Then in section 2 we focus in on our case. We use the root space decomposition of $SO(2n,1)$ to describe the $n$-symmetric space induced by the Coxeter automorphism. We show that a superconformal harmonic tori has a cyclic primitive framing exactly when its harmonic sequence is space-like. This condition is always %satisfied for superconformal harmonic tori in $S^2_1$ and $S^4_1$.
%
%We then prove our finite type result. Firstly, in section 3, by the construction of a Toda frame from a canonical cyclic primitive framing. Secondly, in section 4, by using the Toda frame to show the existence of an adapted polynomial Killing field.
%

It is a pleasure to thank Anthony Henderson for helpful conversations regarding the Lie-theoretic results of section~\ref {dynkin}.

\section{Finite type maps into symmetric spaces}\label{symmetric}

%{\parskip}{\medskipamount}
%\setlength{\parindent}{0pt}

The fact that a harmonic map from a surface to a Lie group corresponds to a loop of flat connections \cite {Pohlmeyer:76,  Uhlenbeck:89} is the fundamental observation that enables one to apply integrable systems techniques to the study of these maps. The Cartan map $ G/H\rightarrow G $ from a symmetric space to the relevant Lie group is well-known to be a totally geodesic immersion when $ G $ is compact and equipped with a bi-invariant Riemannian metric. The composition of a harmonic map with a totally geodesic one is again harmonic, so this enables harmonic maps into symmetric spaces to be studied using the same tools as those into Lie groups, and in particular in terms of a loop of flat connections. We show in Theorem~\ref{thm2.1} 
that when $ G $ has merely a bi-invariant pseudo-metric that the Cartan map is again a totally geodesic immersion. In particular all reductive Lie groups possess a bi-invariant pseudo-metric. We can hence study harmonic maps into  $ G/H $ using integrable systems methods regardless of whether $ G $ is compact.

Let $ G $ be a semisimple Lie group. Recall that a homogeneous space $ G/H $ is a {\it $ k $-symmetric space} ($ k >1 $) if there is an automorphism $\tau: G\rightarrow G $ of order $ k $ such that
\[
(G ^\tau)_0\subset H\subset G ^\tau
\] where $ G ^\tau $ denotes the fixed point set of $\tau $, and $ (G ^\tau)_0 $ the identity component of $ G ^\tau $. When $ k = 2 $, we say that $ G/H $ is a {\it symmetric space}.
% And to additionally distinguish this case, we denote the involution by $\sigma $ instead of $\tau $.
We have the induced action
\begin {align*}
\tau: G/H &\rightarrow G/H\\
 gH &\mapsto\tau (g) H.
\end {align*}
We write $\tau $ also for the induced automorphism of $\g $ and note the $\Z_k $-grading
\[
\g ^\C =\bigoplus_{j = 0} ^ {k -1}\g ^\tau_j,\; [\g ^\tau_j,\g ^\tau_l ]\subset\g ^\tau_{j + l},
\]
 where $\g ^\tau_j $ denotes the $ e ^ {j\frac {2\pi i} {k} } $-eigenspace of $\tau $.

We shall be interested in harmonic maps from a Riemann surface $\Sigma $ into a symmetric space $ G/H $. When $ G $ is compact, the Killing form on $\g $ induces a bi-invariant metric on $ G/H $ and the harmonic map equations for $ f:\Sigma\rightarrow G/H $ may either be calculated directly \cite{Wood:94}, using Noether's Theorem \cite {Rawnsley:84}, % Clinton: could you please add http://www.ams.org.ezproxy1.library.usyd.edu.au/mathscinet-getitem?mr=767837 to my bibliography in research/source; you can download the entry from that list and % then just rename it
or by composing $ f $ with the Cartan map $ G/H\rightarrow G $, which is well-known in this case to be a totally geodesic immersion \cite {CE:75}. Recall here that the Cartan map of a symmetric space is given by
\begin {align*}
\iota:\;\; & G/H\rightarrow G\\
& gH\mapsto \tau (g ) g^{-1}.
\end {align*}
We suppose merely that $ G $ has a bi-invariant pseudo-metric. Then analogous computations hold; in particular we can reduce the problem to studying harmonic maps into the Lie group $ G $ due to the following result.
% If $ G $ is compact, then it has a bi-invariant metric. It is well known that for any compact $ G $ and bi-invariant metric, the mapping $\varphi$ is a totally geodesic embedding \cite{CE:75}.
% We show that this property endures whenever $ G $ has a bi-invariant pseudo-metric.
\begin {theorem}\label{thm2.1}
Let $ G $ be a semisimple  Lie group with bi-invariant pseudo-metric $\langle \cdot,\cdot \rangle $ and $ G/H $ a symmetric space with respect to the involution $\tau: G\rightarrow G $.
Then $\iota: gH\mapsto \tau (g)g ^ {- 1} $ is a totally geodesic immersion $ G/H\rightarrow G $ .
If $ H = G ^\tau $, then $\iota$ is additionally an embedding.
\end {theorem}
%\note { could state with complex bilinear form, but the resulting geodesics would be the same as they are for the pseudometric, so result is no more general}
\noindent Let us call a Lie group $ G $ {\em reductive} if its Lie algebra $\mathfrak g $ is reductive, that is has radical equal to its centre. Then $\mathfrak g $ may be written as the direct sum of a semisimple Lie algebra and an abelian one. On the semisimple Lie algebra the Cartan-Killing form is non-degenerate, whilst on the abelian algebra any bilinear form is invariant under the adjoint action of the group. Combining these we obtain the existence of a bi-invariant pseudometric on any reductive Lie group, and hence the above theorem in particular applies when $ G $ is reductive.
%\note { the symmetric space assumption forces the homogeneous space to be reductive, i.e. $\mathfrak{g} =\mathfrak h\oplus\mathfrak{m} $ with $\mathfrak{m} $ ad-invariant; the group being reductive is an unrelated concept}
\begin {proof}{\em $\iota$ is an immersion:} Suppose $d\iota_{gH} (\gamma' (0)) = 0 $ for some smooth path $\gamma $ in $ G/H $ with $\gamma (0) = gH $. Take a lift $\tilde\gamma $ of $\gamma $ to $ G $ with $\tilde\gamma (0) = g $ and write $\pi: G\rightarrow G/H $ for the projection. Then
\[
0 = % d\iota_{gK}(\gamma'(0))&=&d(\iota\circ\pi)_g(\tilde\gamma'(0))\\&=& (\iota\circ\pi\circ\tilde\gamma)'(0)\\
 \left.\dfrac{d}{dt}\right|_{t=0}\left (\tau\left(\tilde\gamma(t)\right)\left(\tilde\gamma(t)\right)^{-1}\right) = d\tau_g (\tilde\gamma' (0)) g^ {- 1} -\tau (g) g^ {- 1}\tilde\gamma' (0) g^ {- 1},
\]
so
\[
 d\tau_e (g^ {- 1}\tilde\gamma' (0)) =\tau (g^ {- 1}) d\tau_g (\tilde\gamma' (0)) = g^ {- 1}\tilde\gamma' (0)
\]
and $\gamma'(0)$ is zero in $T_{g H}(G/H)$ so $d\iota_{g H}$ is injective.

{\em $\iota$ is totally geodesic:}
 Let $\nabla ^l $ denote the connection on $ G $ obtained by trivialising $ TG $ by left translation, and similarly $\nabla ^ r $ that induced from trivialising by right translation. A computation shows that
$\nabla ^ r =\nabla ^ l +\mathrm{ad}_{g^ {- 1} dg} $ and hence
\[
\nabla =\frac 12 (\nabla ^ l +\nabla ^ r)
\]
is the Levi-Civita connection of the pseudo-metric $\langle \cdot ,\cdot \rangle $.

Denote by $\exp:\mathfrak g\rightarrow G $ the Lie-theoretic exponential map, and by $e $ % Clinton: please declare as a mouse operator
the differential-geometric exponential map associated to the Levi-Civita connection $\nabla $.
Note that as in the definite case, for each $ X\in\mathfrak g $
the map
\begin {align*}
\gamma_X: \g&\rightarrow G\\
 t&\mapsto e ^ {tX}
\end {align*}
\noindent is a geodesic, i.e.
$\nabla_{\gamma'_X}\gamma'_X = 0 $, so $ \exp $ and $e $ agree on the domain of $e $. Since the pseudo-metric is bi-invariant, we conclude that the geodesics through  $ g\in G $ are locally of the form $\gamma (t) = g e ^ {tX} $.
Denote by $\mathfrak{m} $ the $(-1)$-eigenspace of $\tau:\g\rightarrow\g $, and note that $\g =\h\oplus \mathfrak{m} $, where $\h $ is
the Lie algebra of $ H $. The lift $\tilde\gamma (t) = g e ^ {tX} H $ is horizontal, in the sense that $\tilde\gamma' (t)\in g e ^ {tX}\mathfrak{m} $. Thus the geodesics in $ G/H$ through $ g H $ are locally of the form $\tilde\gamma (t) = g e ^ {tX} H $. Since
\[
\iota(g e ^ {tX} H) = g e ^ {tX}\tau (e ^ {- tX})\tau (g^ {- 1}) = g e ^ {2 tX}\tau (g^ {- 1}) = g\tau (g^ {- 1}) e ^ {t\tau (g) X\tau (g^ {- 1})}
\]
is again a geodesic, we conclude that $\iota$ is totally geodesic.

{\em If $ H = G ^\tau $, then $\iota$ is an embedding:} In this case if $\iota(g_1H) =\iota(g_2H) $, then $ g_1 ^ {- 1} g_2 =\tau (g_1 ^ {- 1} g_2) $, and so $ g_1 ^ {- 1} g_2\in H $, and thus $\iota$ is injective.
\end {proof}

%REMOVE BELOW
%The eigenspace decomposition extends to
%\[
%(G/H)\times\g ^\C =\bigoplus_{j = 0} ^ {k -1} [\g_j ]
%\]
%where $ [\g_j ]_{gH} = g\cdot\g_j $ is the $ e ^\frac {2\pi\sqrt {-1}} {k} $-eigenspace of $\Ad_g\cdot\tau\cdot\Ad_{g^{-1}} $.
%REMOVE ABOVE
Let $ F:U\rightarrow G $ be a smooth lift of $ f:U\rightarrow G/H $ on some simply connected $ U\subset\Sigma $, where we assume henceforth that $ G $ is semisimple and has a bi-invariant pseudo-metric (we will later restrict our attention to simple such $ G $.). By the above theorem, $ f $ is harmonic if and only if $\iota\circ f $ is. The Maurer-Cartan form on $ G $ is the unique left-invariant $\g$-valued 1-form which acts as the identity on $\g $. We denote it by $\omega $, and note that if $ G $ is a linear group, then $\omega= g ^ {- 1} dg $. We will use this notation throughout even in the non-linear case.
%\note { Not all  Lie groups admit a finite dimensional faithful linear representation, Fulton and Harris have a counterexample}
Write $ \tilde {f} = \iota\circ f $ and $\Phi = \tilde {f} ^*(\omega) = \tilde {f} ^ {- 1} d\tilde {f} $.  For any smooth $ \tilde {f} $, the form $\Phi $ satisfies the zero-curvature condition
\begin {equation}
d\Phi +\frac 12 [\Phi\wedge\Phi ] = 0,\label {eq:MC}
\end {equation}
known as the Maurer-Cartan equation. Recall that for vector fields $ X, Y $,
\[
[\Phi\wedge\Phi] (X, Y) = 2 [\Phi,\Phi] (X, Y) = [\Phi (X),\Phi (Y)].
\]
% or equivalently, writing $\Phi = A' dz + A''d\bar z $ we have
%\begin {equation}
%A'_{\bar\z} - A"_z = [ A', A" ].
%\end {equation}
The condition that the map $ \tilde {f}:\Sigma\rightarrow G $ is harmonic can be written as
\begin {equation}
d*\Phi = 0.\label {eq:harmonic}
\end {equation}
%or equal to lightly
%\begin {equation}
%A'_{\bar z} + A"_z = 0}
%\end {equation}
%set
%\[
%F^ {- 1} d F = (M' + N') dz + (M" + N") d\bar z.
%\]
Noting that $ \tilde {f} =\tau (F) F^ {- 1} $, we have
\begin {equation}
\label {eq:composesymmetric}
\Phi = F\left (\tau (F)^ {- 1} d (\tau (F)) - F^{-1}d F \right) F^ {- 1} = -2\mathrm{Ad}_F (\varphi_{\mathfrak{m}}),
\end {equation}
where $\varphi  =\varphi _{\mathfrak{h}} +\varphi _{\mathfrak{m}}$ is the decomposition of $\varphi : = F^ {- 1} d F $ into the eigenspaces of $\tau $. Then \eqref {eq:harmonic} becomes
\begin {equation}\label {eq:harmonic1st}
0 = d (\mathrm{Ad}_F (*\varphi _{\mathfrak{m}})) =\mathrm{Ad}_F (d*\varphi _{\mathfrak{m}} + [\varphi \wedge*\varphi_{\mathfrak{m}} ])
\end {equation}
or equivalently,
\begin {equation}
d*\varphi_{\mathfrak{m}} + [\varphi\wedge*\varphi_{\mathfrak{m}} ] = 0.\label{eq:harmonicsymmetric}
\end{equation}
One can also compute the harmonic map equations directly for $ f $. Writing $ [\mathfrak{m}] $ for the subbundle of $ G/H\times\mathfrak{g} $ whose fibre at $ g\cdot x $ is $\mathrm{Ad}_g (\mathfrak{m}) $, we have an isomorphism $ [\mathfrak{m}]\cong {T (G/H)}] $ given by
\begin {align*}
[\mathfrak{m} ]_y &\rightarrow T_y G/H\\
Y &\mapsto\left.\frac {d} {dt}\right|_{t = 0} e ^ {tY}\cdot y.
\end {align*}

The inverse of this isomorphism defines a $\g $-valued 1-form $\theta $ on the symmetric space $ G/H $, which we term its Maurer-Cartan form. Then \cite {Rawnsley:84} $ f $ is harmonic if and only if
\[
d*(f ^*\theta) = 0
\]
and using that
\[
f ^*\theta =\mathrm{Ad}_F (\varphi_{\mathfrak{m}})
\]
we recover  \eqref  {eq:harmonic1st}.
Write $\varphi'_\mathfrak{m} +\varphi_\mathfrak{m}'' $ for the decomposition of $\varphi_\mathfrak{m} $ into $ dz $ and $ d\bar z $ parts. Since $ [\mathfrak{m},\mathfrak{m} ]\subset\h $, a straightforward computation shows \eqref {eq:MC} and \eqref {eq:harmonicsymmetric} are equivalent to the requirement that for each $\lambda\in S ^ 1 $, the form
\begin {equation}\label {eq:form}
%\varphi _\lambda =\lambda^ {- 1}\varphi'_\mathfrak{m} +\varphi_\h +\lambda\varphi''_\mathfrak{m}
\varphi _\lambda =\lambda\varphi'_\mathfrak{m} +\varphi_\h +\lambda^ {- 1}\varphi''_\mathfrak{m}
\end {equation}
satisfies the Maurer-Cartan equation
\begin {equation}
d\varphi_\lambda +\frac 12 [\varphi_\lambda\wedge\varphi_\lambda ] = 0.\label {eq:flat}
\end {equation}
Some solutions to \eqref {eq:flat} can be obtained simply by solving a pair of commuting ordinary differential equations on a finite-dimensional loop algebra. These unusually simple solutions are said to be of  finite type.

 Let $ G/K $ be a $ k $-symmetric space for $ k >2 $ and $\tau $ the corresponding $ k $th order involution. As we shall now explain when mapping into a $ k $-symmetric space for $ k >2 $ it is natural to restrict our attention to a subclass of harmonic maps consisting of those which are primitive, a notion that we now define. Again we have the reductive splitting
\[
\g =\mathfrak{k}\oplus\p
\]
with
\[
\p ^\C =\bigoplus_{j = 1} ^ {k -1}\g_j ^\tau,\qquad \mathfrak {k} ^\C =\g_0 ^\tau.
\]
 Similarly to before we may define the Maurer-Cartan form $\theta $ of the $ k $-symmetric space $ G/K $ when $ k >2 $.  For any smooth lift $ F: U\rightarrow G $ of $ \psi: U\rightarrow G/K $, writing $\varphi = F ^*\omega $ we have
%\begin {equation}\label {eq:primitiveforms}
\[
\psi ^*\theta =\mathrm{Ad}_F\varphi_{\p}.
\]
%\end {equation}
%The Cartan map $\iota (gK) = g\tau (g ^ {-1}) $ is again an immersion $ G/K\rightarrow G $, by the same reasoning as in the symmetric space case. Setting $\kappa =\iota\circ\psi $ then, as in  \eqref  {eq:composesymmetric}, we see that
%\begin {equation}\label {eq:groupprimitive}
%\kappa ^*(\omega) =\kappa^ {- 1} d\kappa =F\left (\tau (F)^ {- 1} d (\tau (F)) - F^{-1}d F \right) F^ {- 1} = (\epsilon -1)\mathrm{Ad}_F (\varphi_{\mathfrak{p}}),
%\end {equation}
%where $\epsilon = e ^ {\frac {2\pi i} {k}} $.
% Combining this with  \eqref {eq:primitiveforms} we see that
%\begin {equation}\label {eq:relationship}
%\kappa ^*(\omega) = (\epsilon -1)

We say that a smooth map $ \psi$ of a surface $\Sigma $ into $ G/K $  is {\em primitive} if the image of $ \psi ^*\theta'  $ is contained in $ [ \g_1 ] $. Equivalently, it is primitive precisely when % $ F ^*\omega'
$\varphi ' = F ^ {- 1}\partial F $ takes values in $\g_0 ^\tau\oplus\g _1^\tau $. Using that $ [\g_1^\tau,\g_{-1} ^\tau]\subset\g_0^\tau $, the Maurer-Cartan equation for $\varphi $ decomposes into $\g_1^\tau $, $\g_0^\tau $ and $\g_{-1} ^\tau $ components as
\begin {align}
d\varphi'_\p + [\varphi _{\mathfrak {k}}\wedge\varphi '_\p] & = 0\label {eq:MCg1}\\
d\varphi_{\mathfrak {k}} + \frac 12 [\varphi _{\mathfrak {k}}\wedge\varphi_{\mathfrak {k}}] + [\varphi '_\p\wedge\varphi''_\p] & = 0\nonumber\\
d\varphi''_\p + [\varphi _{\mathfrak {k}}\wedge\varphi ''_\p] & = 0.\nonumber
\end {align}
% for any lift $ F:\C\rightarrow G $ of $ \psi $, the form $\varphi '_\p $ takes values in $\mathfrak{g}_1 ^\tau $, where $\varphi = F ^*\omega $ and $\varphi' .
%The condition for $\psi $ to be harmonic is that
%\[
%d*(\psi ^*\theta) = d*(\Ad_F (\varphi_\p)) = 0
%\]
From these equations one easily verifies that primitive maps are in particular harmonic. Moreover \cite{BP:94} if $ G/H $ is a symmetric space with $ K\subset H $ and the corresponding reductive splitting preserved under $\tau $, then the projection of $ \psi:\Sigma\rightarrow G/K $ into $ G/H $ is harmonic. An analogous calculation to that above shows that on simply connected subsets $ U\subset\Sigma $, a primitive map $ \psi: U\rightarrow G/K $ is equivalent to a loop
\begin {equation}
\label {eq:primitiveflat}
%\varphi_\lambda =\lambda ^ {- 1}\varphi '_\p +\varphi_{\mathfrak{k}} +\lambda\varphi''_\p,\quad\lambda\in S ^ 1
\varphi_\lambda =\lambda \varphi '_\p +\varphi_{\mathfrak{k}} +\lambda^ {- 1}\varphi''_\p,\quad\lambda\in S ^ 1
\end {equation}
of $\g $-valued 1-forms each satisfying the Maurer-Cartan equation. We see then that both harmonic maps into symmetric spaces and primitive maps into $ k $-symmetric spaces are governed by the same equation \eqref{eq:flat} so we turn now to the question of constructing solutions to this equation.

% Let $ G ^\C $ denote the complexification of the semisimple Lie group $ G $ .
Let $\Omega G $ be the loop group
$
\Omega G =\{\gamma: S ^ 1\rightarrow G \}
$ with corresponding loop algebra
$
\Omega\g: =\{\xi: S ^ 1\rightarrow\g\} $ , where the loops are assumed real analytic without further comment.
We use $\Omega\g ^\C $ to denote loops in the complexified Lie algebra $\g ^\C $.
For studying maps into $ k $-symmetric spaces it is helpful to consider the twisted loop group
\[
\Omega ^\tau G =\{\gamma: S ^ 1\rightarrow G:\gamma (e ^{\frac {2\pi i} {k}\lambda}) =\tau (\gamma (\lambda))\}
\]
 and corresponding twisted loop algebra $\Omega ^\tau\g $ along with its complexification $\Omega ^\tau\g ^\C $.
The (possibly doubly infinite) Laurent expansion
\[
\xi(\lambda) =\sum_{j} \xi_j \lambda  ^ j,\quad\xi_j\in\g ^\tau_j\subset\g ^\C,\quad\Phi_{- j} =\bar\Phi_j
\]
allows us to filtrate $\Omega ^\tau\g ^\C $ by finite-dimensional subspaces
\[
\Omega ^\tau_d =\{\xi\in\Omega\g\mid \xi_j = 0\text { whenever }\left|j\right| >d\}.
\]

Fix a Cartan subalgebra $\mathfrak t$ of $\g $ such that $\mathfrak t\subset\mathfrak k $
and recall that a non-zero $\alpha \in (\mathfrak{t}^\C)^*$ is a {\em root}  with corresponding {\em root space} $\mathcal{G}^{\alpha}\subset\g ^\C $ % = \operatorname{span}\{R_{\alpha}\}$ with $R_{\alpha} \in \mathfrak{g}^\C/ \mathfrak{t}^\C$
 if  $[X_1, X_2 ]=\alpha(X_1) X_2 $ for all $X_1\in \mathfrak{t}$ and $ X_2\in \mathcal{G}^{\alpha}$. % Note that the root spaces are necessarily 1-dimensional.
  We denote the set of roots by $\Delta$ and employ the same notation for the root system formed by considering $\Delta $ as a subset of $ (\mathfrak {t} ^\C) ^*$.
% We can make the decomposition $$\mathfrak{g}^{\mathbb{C}} = \mathfrak{t}^{\mathbb{C}} \oplus \bigoplus_{\alpha \in \Delta}\mathfrak{g}^\alpha.$$
Choose a set of {\em simple roots}, that is a subset $\{ \alpha_1, \ldots, \alpha_N \}$  of $\Delta$ such that every root $\alpha\in \Delta$ can be written uniquely as
$$\alpha=\sum_{j=1}^{N} m_j \alpha_j,$$ where the $m_j$ are either all positive integers or all negative integers. % Write $\Delta ^ + $ and $\Delta ^ - $ for the sets of positive and negative roots respectively.
 The {\em height } of $\alpha$ is $h(\alpha) =\sum_{j = 1} ^ N m_j$ and the root(s) of maximal height are called {\em highest root(s)} whilst those of minimal height are termed {\em lowest root(s)}.

We similarly define the root spaces % $ \mathcal K ^ {\alpha } $
of $\mathfrak k ^\C $. Let $\mathfrak n $ be the nilpotent algebra consisting of the positive root spaces of $\mathfrak k ^\C $ with respect to a choice of simple roots and consider the resulting Iwasawa decomposition
\begin{equation}\label {eq:Iwasawa}
\mathfrak{k}^\C =\mathfrak{n}\oplus\mathfrak{t}^\C\oplus\bar{\mathfrak {n}}
\end {equation}
 of $\mathfrak {k} ^\C $.
Then for $\eta\in\mathfrak k ^\C $ and a local coordinate $ z $ on $\Sigma $, decomposing according to \eqref{eq:Iwasawa} we have
\[
(\eta dz)_{\mathfrak h} = r (\eta) dz +\overline {r (\eta)} d\bar z
\]
where $r:\mathfrak{k}^\C\rightarrow\mathfrak{k}^\C $ is defined by
\begin {equation}\label {eq:r}
r (\eta) =\eta_{\bar{\mathfrak{n}}} +\frac 12\eta_{\mathfrak k}.
\end {equation}

The key observation here is that for simply-connected coordinate neighbourhood $ U\subset\Sigma $, if $\xi: U\rightarrow\Omega ^\tau_d $ satisfies
\begin {equation}\label{eq:flows}
%\dfrac {\partial\xi} {\partial z} = [\xi,\lambda\xi_d + r (\xi_{d -1})]
\dfrac {\partial\xi} {\partial z} = [\xi,\lambda\xi_d + r (\xi_{d -1})]
\end{equation}
then
\[
\varphi_\lambda = (\lambda\xi_d + r (\xi_{d -1})) d z + (\lambda^ {- 1}\xi_{- d} +\overline {r (\xi_{d -1})}) d\bar z
\]
satisfies the Maurer-Cartan equation \eqref{eq:flat} (c.f. \cite{BP:94}, Theorem 2.5). The equation
\[
%\frac 12 ( X (\xi) - iY (\xi)) = (\lambda^ {- 1}\xi_d + r (\xi_{d -1}))
\frac 12 ( X (\xi) - iY (\xi)) = (\lambda\xi_d + r (\xi_{d -1}))
\]
defines vector fields $ X, Y $ on $\Omega_d $. A straightforward computation shows that these vector fields commute and so finding solutions to \eqref {eq:flows} is merely a matter of solving a pair of commuting ordinary differential equations. This yields a rather special class of solutions to the Maurer-Cartan equations \eqref {eq:flat} and hence of harmonic maps to symmetric spaces and primitive maps to $ k $-symmetric spaces, $ k >2 $.
The flows of $ X, Y $ are easily seen to evolve on spheres in $\Omega_d $. When $ G $ is compact, so are these spheres and hence $ X, Y $ are complete and for any initial condition the differential equation \eqref {eq:flows} has a unique solution on $ U $. However when $ G $ is non-compact the completeness of $ X, Y $ is not guaranteed.

\begin {definition}
A harmonic map $ f:\Sigma\rightarrow G/H $ to a symmetric space or a primitive map $ \psi:\Sigma\rightarrow G/K $ to a $ k $-symmetric space, $ k >2 $ is said to be of {\em finite type} if it has a lift $ F:\Sigma\rightarrow G $ for which there exists
a smooth map $\xi: \R \to\Omega_d ^\tau\g$ satisfying
\begin {equation}\label {eq:lax}
d\xi = [\xi,\varphi _\lambda ]
\end {equation}
and
\begin {equation}\label {eq:adapted}
%\varphi_\lambda = (\lambda^ {- 1}\xi_d + r (\xi_{d -1})) d z + (\lambda\xi_{- d} +\overline {r (\xi_{d -1})}) d\bar z,
\varphi_\lambda = (\lambda\xi_d + r (\xi_{d -1})) d z + (\lambda^ {- 1}\xi_{- d} +\overline {r (\xi_{d -1})}) d\bar z.
\end {equation}
Here $\varphi_\lambda $ and $ r $ are defined in  \eqref  {eq:primitiveflat}
 and \eqref {eq:r} for the primitive case and in \eqref{eq:form} and the obvious analogue to  \eqref {eq:r} for the harmonic case.
\end {definition}
We introduce some terminology for later use. A {\em formal Killing field} for $ f $ or $ \psi $ is a smooth map $\xi:\Sigma\rightarrow\Omega ^\tau\mathfrak g $ satisfying the Lax equation \eqref{eq:lax}. When $\xi $ takes values in some $\Omega_d$ it is termed a {\em polynomial Killing field} of degree $ d $  and when it additionally satisfies \eqref{eq:adapted} it is an {\em adapted polynomial Killing field}.

When the automorphism $\tau: \g ^\C\rightarrow \g ^\C $ is of the form $\tau =\mathrm{Ad}_{\exp M}$ for some
% $ M\in\mathfrak{g}  $ then we say that $ G/H $ or $ G/K $ is an inner  ($ k $)-symmetric space. We shall be interested in a slightly different situation, namely when
$ M\in\mathfrak t ^\C $ where $\mathfrak t $ is a Cartan subalgebra of $\mathfrak g $,
  then we can express the eigenspaces $\mathfrak{g} ^\tau_j $ of $\tau $ in terms of root spaces. % We henceforth assume that the automorphism $\tau $ is of this form.

Given our chosen set of simple roots $\alpha _j $, denote by $ \eta_j$ the corresponding dual basis of $\mathfrak{t} ^\C$.
%\todo {sort out whether should have dual basis of $ i\t $ as when Killing form negative definite, or}
For any root  $\alpha=m_1\alpha_1 + \ldots m_N \alpha_N$, smooth map $s_j:\Sigma \to \C$ and root vector $R_\alpha\in\mathcal{G} ^\alpha $, a straightforward computation shows that
\begin {equation}
\Ad_{ \exp(s_1 \eta_1+\ldots s_N \eta_N)} R_\alpha = \exp(m_1 s_1+\ldots m_Ns_N)R_\alpha.\label {eq:roots}
\end {equation}
Note that $ \exp(m_1 s_1+\ldots m_Ns_N)$ is a scalar function.
%\begin{proof} From the nature of exponentials we only need to consider each $1\leq j \leq N$. By construction of $\eta_j$ we have $[s_j\eta_j,R_\alpha] =m_j s_j  R_\alpha. $
%Thus
%\begin{align*}
%&\Ad \exp(s_j \eta_j) R_\alpha \\
%&= R_\alpha+ [s_j \eta_j,R_\alpha] +\frac{1}{2}[s_j\eta_j,  [s_j\eta_j,R_\alpha] ] + \frac{1}{3!}[s_j \eta_j,[s_j \eta_j ,[s_j\eta_j,R_\alpha]]] + \ldots\\
%&= R_\alpha +  m_j s_j R_\alpha +  \frac{1}{2}\left( m_j s_j\right)^2 R_\alpha + \frac{1}{3!}\left(m_j s_j \right)^3 R_\alpha + \ldots\\
%&=\exp\left( m_j s_j\right)R_\alpha.
%\qedhere
%\end{align*}
%\end{proof}
%
Given $\tau = \Ad_{ \exp (\frac{2\pi i}{k}(\sum s_j \eta_j))} $ we have
$$\mathfrak{g}^\tau_l = \Span \{R_\alpha| \alpha=\sum_{j = 1} ^ N m_j\alpha_j, \sum_{j = 1} ^ N s_jm_j=l\bmod(k)\}.$$
In particular if we let $k-1$ denote the maximal height of a root of $\g ^\C $ and % take $Z:=\frac{1}{k}\sum_{j=1}^N \eta_j$ and
suppose
\begin {equation}\label {eq:Coxeterdefinition}\sigma:=\Ad_{ \exp ( \frac {2\pi i } {k}\sum_{j = 1} ^ N \eta_j)}\,,
\end {equation}
 then $\sigma$ is of order $k$  and from \eqref {eq:roots} it acts on the root spaces by
\begin {equation}\label {eq:Coxeter}
\sigma(R_\alpha) = \exp\left(\frac{2\pi i h(\alpha)}{k}\right)R_\alpha.
\end {equation}
We recognise the  inner automorphism $\sigma$ as the Coxeter automorphism associated to the identity transformation of the simple roots  \cite {BD:81}. It plays an important role here because when it preserves the real  Lie group $ G $,
% If $\sigma $ preserves $\g $, then
 it allows us to view $G/T$ as a $k$-symmetric space for which $\g ^\sigma_1 $ %, primitive maps are those who have a lift $F$ such that $F^{-1} F_z$ lies in $\mathfrak{t}^\C$ and
is the sum of the simple and lowest root spaces. Here $ T $ is a Cartan subgroup with Lie algebra $\mathfrak t $.  Furthermore since $ K = T $ in this case, the map $ r $ described in \eqref {eq:r} is simply multiplication by $\tfrac {1} {2} $ and so the adapted polynomial Killing field condition \eqref {eq:adapted} simplifies. Taking this $ N $-symmetric space structure on $ G/T $, we say that a smooth map $ \psi:\Sigma\rightarrow G/T $ is  \emph{cyclic primitive} if it is primitive and satisfies the  condition that the image of $ \psi ^*\theta' $ contains a cyclic element. Writing $\alpha_0 $ for the lowest root, an element in $\left (\bigoplus_{j = 0} ^ N \mathcal G ^{\alpha_j}\right) $ is  \emph{cyclic} if its projection to each of the root spaces $  \mathcal G ^{\alpha_0},  \mathcal G ^{\alpha_1},\ldots ,  \mathcal G ^{\alpha _N} $ is non-zero. We henceforth assume that $ G $ is simple in order to guarantee the uniqueness of the lowest root (that is, we assume that $ G $ is connected and $\g $ is simple).

\section {Extended Dynkin diagrams and Cartan involutions}\label{dynkin}
We now ascertain the $ k $-symmetric spaces to which our theory will apply. That is, we give conditions under which a choice of real form $\g $ of a simple complex Lie algebra $\g ^\C $, Cartan subalgebra $\mathfrak t ^\C $ and simple roots $\alpha_j $ % and corresponding root vectors $ R_{\alpha_j} $
yield a Coxeter automorphism $\sigma $ which preserves the real Lie algebra $\g $. When $ G ^\C $ is a simply connected or adjoint simple Lie group with Lie algebra $\g ^\C $, this ensures that the Coxeter automorphism preserves the real group $ G $.
% Automorphism always lifts to simply connected form of group and preserves centre, hence lifts to adjoint form = simply connected form/centre

Let $\;\bar{}\;$
denote the complex conjugation of $\g ^\C $ corresponding to the real form $\g $. Define the conjugate of a root $\alpha $ by
\[
\bar\alpha (X) =\overline {\alpha (\bar X)}.
\]
Then from \eqref {eq:Coxeter} we see that the condition for the Coxeter automorphism $\sigma $ to preserve $\g $ is that for all roots $\alpha $, the height $ h (\alpha $) satisfies
\[
h (\bar\alpha) = - h (\alpha)\bmod k,
\]
or equivalently that for
$ j = 1,\ldots, N $ we have
\begin {equation*}\label {eq:simplereality}
\bar\alpha_j\in\{-\alpha_0,\ldots, -\alpha_N\}.
\end {equation*}
We will now use a Cartan involution to express this reality condition in terms of the extended Dynkin diagram for $\alpha_0,\ldots,\alpha_N $. A Cartan involution is an involution $\Theta $ of $\g $ such that
\[
 \langle X, Y \rangle _{\Theta} = - \langle X,\Theta (Y) \rangle
\]
is positive definite, where $ \langle \cdot ,\cdot  \rangle $ denotes the Killing form. Using complex-linearity, $\Theta $ extends to an involution of $\g ^\C $.
We may \cite[Prop.\! 6.59]{Knapp:02} % Clinton: could you please look up how to make the spacing following the full stop After prop not the same as for a new sentence, and also look up how to give chapters within citations so that within the file it % looks like [24, Prop. 6.59] rather than is currently
choose a Cartan involution $\Theta $ which preserves the Cartan subalgebra $\mathfrak t $.
\begin {proposition}\label {prop:Coxeter}
Let $\g $ be a real simple Lie algebra, $\t $ a Cartan subalgebra and $\Theta $ be a Cartan involution preserving $\mathfrak t $. Choose simple roots $\alpha_1,\ldots,\alpha_N $ for the root system $\Delta (\g ^\C,\t ^\C) $ and let $\sigma $ be the corresponding Coxeter automorphism of $\g ^\C $ defined in \eqref {eq:Coxeterdefinition}. Then the following are equivalent:
\begin {enumerate}
\item  $\sigma $ preserves the real form $\g $,
\item $\sigma $ commutes with $\Theta $,
\item $\Theta $ defines a permutation of the extended Dynkin diagram for $\g ^\C $ consisting of the usual Dynkin diagram augmented with the lowest root $\alpha_0 $.
\end {enumerate}
\end {proposition}

\begin {proof} Write $\mathfrak t =\mathfrak l\oplus \mathfrak p $, where $\mathfrak l $, $\mathfrak p $ are respectively the $ (+ 1 ) $-eigenspace and $ (- 1) $-eigenspace of the action of $\Theta $ on $\mathfrak t $.
Then  \cite[Cor. 6.49] {Knapp:02}  all roots $\alpha $ are real on $\mathfrak p $ and imaginary on $\mathfrak l $, and defining the action of $\Theta $ on roots by $\Theta (\alpha) (X) =\alpha (\Theta (X)) $ we have that
\[
\Theta (\alpha) = -\bar\alpha\quad\text { for all roots }\alpha.
\]
If $ R_\alpha $ is a root vector for $\alpha $, then $\bar R_\alpha $ is a root vector for $\bar\alpha $ and $\Theta (R_\alpha) $ is a root vector for $\Theta (\alpha) $. We  assume that our root vectors are chosen so that
\[
R_{\bar\alpha} =\bar R_\alpha
\]
and write $ R_{\Theta (\alpha)} = c_\alpha\Theta (R_\alpha) $. Then using \eqref {eq:Coxeter}, a straightforward computation shows that $\sigma\circ\Theta (R_\alpha) =\Theta\circ\sigma (R_\alpha) $ if and only if $\sigma (\bar R_{-\alpha}) =\overline {\sigma (R_{-\alpha})} $, % Clinton he was always if they don't specifically say its capital, Sigma is not I
proving the equivalence of conditions (1) and (2) above. % Clinton: please format the one and 2 correctly, there might be India and double sign i and the brackets probably outright % $\Sigma $ preserves $\golf $ if and only if you commute to $\Theta $.

The Cartan involution $\Theta $ commutes with $\sigma $ if and only if for all roots $\alpha $, the height function $ h $ satisfies
\[
h (\Theta (\alpha))\equiv h (\alpha)\bmod k,
\]
or equivalently when $\Theta $ defines a permutation of
$\alpha_0,\alpha_1,\ldots,\alpha_N $. All automorphisms of a Lie algebra preserve the Killing form and hence a Cartan involution $\Theta $ as above defines a permutation of the extended Dynkin diagram
 and we see the equivalence of conditions (2) and (3).
\end {proof}
 We next show that every involution of the extended Dynkin diagram for $\Delta (\g ^\C,\t ^\C) $ does indeed arise from a Cartan involution for some real form $\g $ with $\Theta $-stable Cartan subalgebra $\t =\g\cap\t ^\C $.
A $\Theta $-stable Cartan subalgebra $\t $ of $\g $ is maximally compact if and only if $\Theta $ preserves a set of simple roots for the root system $\Delta (\g,\t)$ \cite [p 387] {Knapp:02}. Hence when $\t $ is maximally compact, a Coxeter automorphism $\sigma $ must stabilise the real form $\g $. (In particular, all Cartan subalgebras of a compact real form $\g $ are maximally compact.) The more interesting case then is when the Cartan subalgebra $\t $ is not maximally compact, which corresponds to the involution of the extended Dynkin diagram acting nontrivially on the lowest root $\alpha_0 $.
% This prompts us to ask {\it for which root systems $\Delta $ can we find a real Lie algebra $\g$, Cartan subalgebra $\t $ and simple roots  such that $\t $ is not maximally compact but the corresponding Coxeter automorphism preserves $\g $?} We show next that it suffices

\begin{theorem}\label{thm:3.2} Every involution of the extended Dynkin diagram for a simple complex Lie algebra $\g ^\C $ is induced by a Cartan involution of a real form of $\g ^\C $.

 More precisely, let $\g^\C $ be a simple complex Lie algebra with Cartan subalgebra $\t ^\C $ and choose simple roots $\alpha_1,\ldots,\alpha_N $ for the root system $\Delta (\g ^\C,\t ^\C) $. Given an involution $\pi $ of the extended Dynkin diagram for $\Delta $,
%, a Cartan subalgebra $\t ^\C $ of $\g ^\C$ and simple roots $\alpha_1,\ldots,\alpha_N $,
there exists  a real form $\g $ of $\g ^\C $ and a Cartan involution $\Theta $ of $\g $ preserving $\t =\g\cap\t ^\C $ such that $\Theta $ induces $\pi $ and $\t $ is a real form of $\t ^\C $. The Coxeter automorphism $\sigma $ determined by $\alpha_1,\ldots,\alpha_N $ preserves the real form $\g $.
\end {theorem}
\begin {proof}
Let $\pi $ be an involution of the extended Dynkin diagram. Denote also by $\pi $ the corresponding involution of the set $\{0, 1,\ldots, N\} $ and the induced   involution of $ (\t ^\C) ^* $ which preserves the root system $\Delta$ and satisfies $\pi (\alpha_j) =\alpha_{\pi (j)} $. 
% Any choice of root vectors $ R_{\alpha_j} $ for $ j = 0,\ldots, N $ and

Let $\{H_\alpha, R_\alpha\mid\alpha\in\Delta\}$ be a Chevalley basis. That is, writing $\alpha ^\# $ for the dual of the root $\alpha $ with respect to the Killing form $\kappa$ we set $ H_\alpha = (2/ \kappa(\alpha ^\#,\alpha ^\#)){\alpha ^\#} $ and we choose the root vectors $ R_\alpha $ so that
\[
[R_\alpha, R_{-\alpha}] = H_\alpha.
\]
and such that the {\em structure constants} $ c_{\alpha,\beta} $ defined by $ [R_\alpha, R_\beta] = c_{\alpha,\beta} R_{\alpha +\beta} $ satisfy 
$ c_{-\alpha, -\beta} = - c_{\alpha,\beta} $.
%\begin {equation}\label {eq:structure}
%\end {equation}
%\begin {equation}\label {eq:negative_structure}
%\end {equation}
%For the classical Lie algebras, it is well-known that there is a subgroup of the automorphism group which is isomorphic to the Weyl group, \verify {is this true for $ E_7 $?} and hence to verify that $\Theta $ is an involution it suffices to show that the %corresponding $\pi $ lies in the Weyl group. The quotient of the group of automorphisms of the root system $\Delta$ by the Weyl group is isomorphic to the group of automorphisms of the Dynkin diagram, which we can see is trivial for $ B_N, C_N $ and %$ E_7 $.
For any $ b_{\alpha_j}\in\C $ for $ j = 1,\ldots, N $, we obtain  an automorphism $\Theta $ of $\g ^\C $ compatible with $\pi $ by requiring that $\Theta (R_{\alpha_j}) = b_{\alpha_j} R_{\pi (\alpha_j)}$ for $ j = 1, \ldots, N $ and that $\{\pi (H_\alpha),\Theta (R_\alpha)\mid\alpha\in\Delta\} $ is a Chevalley basis. 
Our first task is to verify that for an appropriate choice of $ b_{\alpha_j} $, the resulting $\Theta $ is an involution. 

Given $\pi $ and $ b_{\alpha_1}, \ldots , b_{\alpha_N} $, for any root $\alpha $ we  define $ b_{\alpha}\in\C $ by the equation $\Theta (R_{\alpha}) = b_{\alpha}R_{\pi (\alpha)} $. 
The automorphism $\Theta $ will be an involution precisely when $ b_{\alpha_j}b_{\alpha_{\pi (j)}} = 1 $ for $ j = 1, \ldots, N $. For the $j$ with $\pi (j)\neq 0$, we can clearly guarantee this by taking $ b_{\alpha_{\pi (j)}} = b_{\alpha_j}^ {- 1} $.
% When $\pi (j)\neq j $. This is achieved by choosing $ b_{\alpha_j} =\pm 1 $ when $\pi( j)=h $ and $b_{\alpha_j}=1$ when $\pi(j) \neq j$). 
We will show that $ b_{\alpha_1}, \ldots , b_{\alpha_N} $ can be chosen so that additionally $ b_{\alpha_0}b_{\alpha_{\pi (0)}} = 1 $.

% From the Dynkin diagram we can read off the extended Cartan matrix which necessarily has determinant 0 and has a unique 0-eigenvector of the form $ (1, m_1,\ldots, m_N) $. The lowest root is given by $\alpha_0 = -\sum_{j= 1} ^ Nm_j\alpha_j $, and we %see that the $ m_j $ are preserved by an involution of the extended Dynkin diagram. Hence an involution of the extended Dynkin diagram necessarily comes from an automorphism.

We may express $R_{\alpha_0} $ as 
$C[R_{-\beta_1}, [R_{-\beta_2}, \ldots, [R_{-\beta_{K-1}}, R_{-\beta_K}]\ldots]]$ for some non-zero constant $C$ and simple roots $\beta_i$  such that $\sum_{i = 1} ^ K \beta_i = -\alpha_0$.
Now writing $\alpha_0=-\sum_{j = 1} ^ N m_j\alpha_j$ we have
\begin{equation}\label{eq:involbj}
b_{\alpha_0} R_{\alpha_{\pi (0)} =\Theta (R_{\alpha_0}}) = C\prod_{j=1}^N b_{-\alpha_j}^{m_j}[R_{-\pi(\beta_1)}, [R_{-\pi(\beta_2)}, \ldots, [R_{ \pi(\beta_{K-1})}, R_{-\pi(\beta_K)}]\ldots]]
\end{equation}
and $\Theta^2(R_{\alpha_0}) =\prod_{j=1}^N (b_{-\alpha_j}b_{-\alpha_{\pi (j)}})^{m_j} R_{\alpha_0}$, implying 
%\begin{align}\label{eq:involb0bpi0}
%b_{\alpha_0} b_{\alpha_{\pi (0)}} = \Pi_{j=1}^N (b_{-\alpha_j}b_{-\alpha_{\pi (j)}})^{m_j}.
%\end{align}
 % Applying $\Theta ^ 2 $ yields 
\[
b_{\alpha_0} b_{\alpha_{\pi (0)}} = \prod_{j=1}^N (b_{-\alpha_j}b_{-\alpha_{\pi (j)}})^{m_j}.
\]
% Applying $\Theta ^ 2 $ yields 
%\[
%b_{\alpha_0} b_{\alpha_{\pi (0)}} = \prod_{j=1}^N (b_{-\alpha_j}b_{-\alpha_{\pi (j)}})^{m_j},
%\]
%where $\alpha_0=-\sum_{j = 1} ^ N m_j\alpha_j$.
 % Defining $ b_{-\alpha_j} $, $ j = 1,\ldots, N $ by $\Theta (R_{-\alpha_j}) = b_{-\alpha_j} R_{\pi (-\alpha_j)} $ and 
Using that $\{\pi (H_{\alpha}),\Theta (R_{\alpha})\mid\alpha\in\Delta\} $ is again a Chevalley basis and that an automorphism of the extended Dynkin diagram must preserve the Killing form gives
\[
b_{\alpha_j}b_{-\alpha_j} =\frac {\kappa (\pi (\alpha_j)) , \pi (\alpha_j)) )} {\kappa (\alpha_j ,\alpha_j )}= 1.
\]
Hence 
\[
 b_{\alpha_0}b_{\alpha_{\pi (0)}} =\prod_{j = 1} ^ N (b_{\alpha_j}b_{\alpha_{\pi (j)}}) ^ {- m_j} = (b_{\alpha_{\pi (0)}} b_{\alpha_0}) ^ {- 1},
\]
where the last equality uses the assumption $ b_{\alpha_j}b_{\alpha_{\pi (j)}} = 1 $ for $\pi (j)\neq 0 $.
We therefore automatically have $ b_{\alpha_0}b_{\alpha_{\pi (0)}} =\pm 1 $. % If there exists $ j $ such that $\pi (j) = j $ and $ m_j $ is odd then by switching the sign of $ b_{\alpha_j} $ if necessary we may ensure that $ b_{\alpha_0}b_{\alpha_{\pi (0)}} = 1 $.
Considering \eqref{eq:involbj} shows that if there exists $ j $ such that $\pi (j) = j $ and $ m_j $ is odd then by switching the sign of $ b_{\alpha_j} $ if necessary we may ensure that $ b_{\alpha_0}b_{\alpha_{\pi (0)}} = 1 $.

It remains to give a method of proof for when there is no  $\alpha_j$ with $m_j$ odd that is fixed by $\pi $. If $\pi (0) = 0 $ then there is nothing to prove so we assume henceforth that $\pi (0)\neq 0 $. Suppose $\gamma$ is a positive root such that
\begin {enumerate} [(a)] % added enumerate package
\item the expression $\gamma =\sum_{j = 1} ^ N n_j\alpha_j $ as a sum of simple roots has $ n_{ \pi(0)} = 0 $,
\item $\pi(\gamma) + \alpha_{\pi(0)}$ is also a root, and 
\item $\gamma + \alpha_0 =- \pi(\gamma) - {\alpha_{\pi (0)}} $.
\end {enumerate}
 % Now $[R_\gamma,R_{\alpha_0}}] = c_{\gamma,\alpha_0} R_{\gamma+\alpha_0}$ and $[R_{\pi(\gamma)}, R_{\pi(0)}]= c_{\pi (\gamma),\pi (\alpha_0)} R_{\pi(\gamma) + \alpha_{\pi(0)}}$. 
% Since $\pi(\gamma + \alpha_0)=-(\gamma+\alpha_0)$ we know that 
From (c) we have that
\[
[[R_\gamma,R_{\alpha_0}],[R_{\pi(\gamma)}, R_{\pi(0)}]] =  c _{\gamma,\alpha_0}c_{\pi (\gamma),\pi (\alpha_0)} H_{\gamma + \alpha_0}.
\]
Applying $\Theta$  gives 
\[
[[b_\gamma R_{\pi(\gamma)},b_{\alpha_0}R_{\pi(0)}],[b_{\pi(\gamma)}R_\gamma, b_{\alpha_{\pi (0)}}R_{\alpha_0}]] = - c _{\gamma,\alpha_0}c_{\pi (\gamma),\pi (\alpha_0)} H_{\gamma + \alpha_0}
\]
and so
\begin{align}\label{eq:bgamma}
b_\gamma b_{\pi(\gamma)}b_{\alpha_0}b_{\alpha_{\pi (0)}} = 1.
\end{align}
We may write $R_\gamma$ as $C'[R_{\beta'_1},[R_{\beta'_2} \ldots  [R_{\beta'_{K'-1}}, R_{\beta'_{K'}}]]\ldots  ]$ with $C'$ a non-zero constant and $\beta'_i\neq\alpha_{\pi(0)}$ simple roots satisfying $\sum_{i = 1} ^ {K'}\beta'_i =\gamma$. Then 
\[
b_\gamma b_{\pi (\gamma)} R_\gamma =\Theta^2(R_\gamma) =\left (\prod_{i = 1} ^ {K'} b_{\beta'_i}b_{\beta'_{\pi(i)}}\right) R_\gamma. 
\]
However for simple roots $\alpha_j $ with $\pi (j)\neq 0 $ we chose $b_{\alpha_j}$ so that $ b_{\alpha_j} b_{\alpha_{\pi (j)}} = 1 $ and hence % $b_{\beta_i}b_{\beta_{\pi(i)}}=1$ for all $i$ and 
 $b_\gamma b_{\pi(\gamma)}=1$. Substituting this into \eqref{eq:bgamma} gives that $b_{\alpha_0}b_{\alpha_{\pi (0)}} = 1$, as required. 

A similar argument applies if there are positive roots $\gamma, \delta$ such that
\begin {enumerate} [(i)]
\item  the expressions of $\gamma,\delta $ as sums of simple roots do not contain $\alpha_{\pi(0)}$,
\item $\pi(\gamma) + \alpha_{\pi(0)}$ and $\delta + \pi(\delta)$ are also roots, and
\item $\delta +\pi(\delta) + \gamma + \pi(\gamma) = - \alpha_0 - \alpha_{\pi(0)} $. 
\end {enumerate}
Here we know there is some non-zero constant $C'' $ such that 
$$[[R_\gamma,R_{\alpha_0}],[R_{\pi(\gamma)}, R_{\pi(0)}]] =  C'' [R_{-\delta}, R_{-\pi(\delta)}]$$ and as above applying $\Theta$ gives
% $$-b_\gamma b_{\pi(\gamma)}b_{\alpha_0}b_{\alpha_{\pi (0)}}[[R_\gamma,R_{\alpha_0}}],[R_{\pi(\gamma)}, R_{\pi(0)}]] = - b_{-\delta} b_{-\pi(\delta)} C[R_{-\delta}, R_{-\pi(\delta)}]$$ and hence 
\[
b_\gamma b_{\pi(\gamma)}b_{\alpha_0}b_{\alpha_{\pi (0)}} = b_{-\delta} b_{-\pi(\delta)}.
\]
By (i) % Since $\alpha_{\pi(0)}$ is not contained in the sum of simple roots of either $\gamma$ or $\delta$ 
we know $b_\gamma b_{\pi(\gamma)} = 1$ and  $b_{-\delta} b_{-\pi(\delta)}=1$ so conclude that
$b_{\alpha_{\pi (0)}}b_{\alpha_0}=1$.

To show that every involution of the extended Dynkin diagram extends to an involution of the Lie algebra we now consider the involutions of each of the diagrams and, for those that do not fix some $\alpha_j$ with odd $m_j$, identify a suitable root $\gamma$ or pair of roots $\gamma, \delta$. 

% And hence for any given involution $\pi $ of the extended Dynkin diagram an appropriate choice of $ b_{\alpha_j} $, $ j = 1,\ldots, N $ yields an involution $\Theta $ of $\g ^\C $.
% Clinton: could you please change the A_N diagram by making it look like a circle (I want to be able to talk about reflections of the diagram). Probably space was this means that the classical Lie algebras will be in the left column, in which case you might as well reverse the order of % the E diagrams so that they are e_6 through to E_8
\setcounter{figure}{0}
\begin{figure}[h]\label{figure:extended}
\begin{tikzpicture}[scale=0.8]
\draw(0,0) node {$E_8$};
\filldraw(1,0) circle (3pt);
\draw(2,0) circle (3pt);
\draw(3,0) circle (3pt);
\draw(4,0) circle (3pt);
\draw(5,0) circle (3pt);
\draw(6,0) circle (3pt);
\draw(7,0) circle (3pt);
\draw(8,0) circle (3pt);
\draw(6,1) circle (3pt);
\draw(1.10,0)--(1.90,0);
\draw(2.10,0)--(2.90,0);
\draw(3.10,0)--(3.90,0);
\draw(4.10,0)--(4.90,0);
\draw(5.10,0)--(5.90,0);
\draw(6.10,0)--(6.90,0);
\draw(7.10,0)--(7.90,0);
\draw(6,0.10)--(6,0.90);
\draw(1,-0.6) node {\footnotesize{$\alpha_0$}};
\draw(2,-0.6) node {\footnotesize{$\alpha_8$}};
\draw(3,-0.6) node {\footnotesize{$\alpha_7$}};
\draw(4,-0.6) node {\footnotesize{$\alpha_6$}};
\draw(5,-0.6) node {\footnotesize{$\alpha_5$}};
\draw(6,-0.6) node {\footnotesize{$\alpha_4$}};
\draw(7,-0.6) node {\footnotesize{$\alpha_3$}};
\draw(8,-0.6) node {\footnotesize{$\alpha_1$}};
\draw(5.5,1) node {\footnotesize{$\alpha_2$}};

\draw(0,2) node {$D_N$};
\draw(1,2) circle (3pt);
\draw(2,2) circle (3pt);
\draw(4,2) circle (3pt);
\draw(5,2) circle (3pt);
\filldraw(2,3) circle (3pt);
\draw(4,3) circle (3pt);
\draw(1.10,2)--(1.90,2);
\draw(2.10,2)--(2.60,2);
\draw(3.40,2)--(3.90,2);
\draw(4.10,2)--(4.90,2);
\draw(3,2) node{$\ldots$};
\draw(2,2.10)--(2,2.90);
\draw(4,2.10)--(4,2.90);
\draw(1,1.6) node {\footnotesize{$\alpha_1$}};
\draw(2,1.6) node {\footnotesize{$\alpha_2$}};
\draw(1.5,3) node {\footnotesize{$\alpha_0$}};
\draw(4,1.6) node {\footnotesize{$\alpha_{N-2}$}};
\draw(3.3,3) node {\footnotesize{$\alpha_{N-1}$}};
\draw(5,1.6) node {\footnotesize{$\alpha_N$}};

\draw(0,4) node {$C_N$};
\filldraw(1,4) circle (3pt);
\draw(2,4) circle (3pt);
\draw(3,4) circle (3pt);
\draw(5,4) circle (3pt);
\draw(6,4) circle (3pt);
\draw[double](1.10,4)--(1.90,4);
\draw(1.80,3.90)--(1.90,4);
\draw(1.80,4.10)--(1.90,4);
\draw(2.10,4)--(2.90,4);
\draw(3.10,4)--(3.60,4);
\draw(4.40,4)--(4.90,4);
\draw(4,4) node {$\ldots$};
\draw[double](5.10,4)--(5.90,4);
\draw(5.20,4.10)--(5.10,4);
\draw(5.20,3.90)--(5.10,4);
\draw(1,3.6) node {\footnotesize{$\alpha_0$}};
\draw(2,3.6) node {\footnotesize{$\alpha_1$}};
\draw(3,3.6) node {\footnotesize{$\alpha_2$}};
\draw(5,3.6) node {\footnotesize{$\alpha_{N-1}$}};
\draw(6,3.6) node {\footnotesize{$\alpha_N$}};

\draw(0,6) node {$B_N$};
\draw(1,6) circle (3pt);
\draw(2,6) circle (3pt);
\draw(4,6) circle (3pt);
\draw(5,6) circle (3pt);
\filldraw(2,7) circle (3pt);
\draw(1.10,6)--(1.90,6);
\draw(2.10,6)--(2.60,6);
\draw(3.40,6)--(3.9,6);
\draw[double](4.1,6)--(4.9,6);
\draw(4.8,6.1)--(4.9,6);
\draw(4.8,5.9)--(4.9,6);
\draw(2,6.10)--(2,6.9);
\draw(3,6)node {$\ldots$};
\draw(1,5.6) node {\footnotesize{$\alpha_1$}};
\draw(2,5.6) node {\footnotesize{$\alpha_2$}};
\draw(4,5.6) node {\footnotesize{$\alpha_{N-1}$}};
\draw(5,5.6) node {\footnotesize{$\alpha_N$}};
\draw(1.5,7) node {\footnotesize{$\alpha_0$}};

%\draw(0,8) node {$A_N$};
%\draw(1,8) circle (3pt);
%\draw(1,7.6) node {\footnotesize{$\alpha_1$}};
%\draw(2,8) circle (3pt);
%\draw(2,7.6) node {\footnotesize{$\alpha_2$}};
%\draw(4,8) circle (3pt);
%\draw(4,7.6) node {\footnotesize{$\alpha_{N-1}$}};
%\draw(5,8) circle (3pt);
%\draw(5,7.6) node {\footnotesize{$\alpha_N$}};
%\filldraw(3,9) circle (3pt);
%\draw(3,9.4) node {\footnotesize{$\alpha_0$}};
%\draw(1.1,8)--(1.9,8);
%\draw(2.1,8)--(2.6,8);
%\draw(3.4,8)--(3.9,8);
%\draw(4.1,8)--(4.9,8);
%\draw(3,8) node {$\ldots$};
%\draw(1.05,8.1)--(3,9);
%\draw(3,9)--(4.95,8.1);
\draw(2.658,8.060)arc(-110:-359:1cm);
\draw(4,9)arc(0:-70:1cm);
\draw(3,8)node{$\ldots$};
\draw(0,9)node{$A_N$};
\draw(1.8,9.8)node{\footnotesize{$\alpha_1$}};
\draw(4.3,9.8)node{\footnotesize{$\alpha_N$}};
\draw(4.8,9)node{\footnotesize{$\alpha_{N-1}$}};
\draw(1.6,9)node{\footnotesize{$\alpha_2$}};
\draw(3,10.4)node{\footnotesize{$\alpha_0$}};
\filldraw(3,10)circle(3pt);
\filldraw[color=white](3.707,9.707)circle(3pt);
\filldraw[color=white](4,9)circle(3pt);
\filldraw[color=white](2.293,9.707)circle(3pt);
\filldraw[color=white](2,9)circle(3pt);

\draw(3.707,9.707)circle(3pt);
\draw(4,9)circle(3pt);
\draw(2.293,9.707)circle(3pt);
\draw(2,9)circle(3pt);

\draw(9,8) node {$E_7$};
\draw(10,8) circle (3pt);
\draw(11,8) circle (3pt);
\draw(12,8) circle (3pt);
\draw(13,8) circle (3pt);
\draw(14,8) circle (3pt);
\draw(15,8) circle (3pt);
\filldraw(16,8) circle (3pt);
\draw(10.1,8)--(10.9,8);
\draw(11.1,8)--(11.9,8);
\draw(12.1,8)--(12.9,8);
\draw(13.1,8)--(13.9,8);
\draw(14.1,8)--(14.9,8);
\draw(15.1,8)--(15.9,8);
\draw(13,8.1)--(13,8.9);
\draw(13,9) circle (3pt);
\draw(10,7.6) node {\footnotesize{$\alpha_7$}};
\draw(11,7.6) node {\footnotesize{$\alpha_6$}};
\draw(12,7.6) node {\footnotesize{$\alpha_5$}};
\draw(13,7.6) node {\footnotesize{$\alpha_4$}};
\draw(14,7.6) node {\footnotesize{$\alpha_3$}};
\draw(15,7.6) node {\footnotesize{$\alpha_1$}};
\draw(16,7.6) node {\footnotesize{$\alpha_0$}};
\draw(12.5,9) node {\footnotesize{$\alpha_2$}};

\draw(9,5) node {$E_6$};
\draw(10,5) circle (3pt);
\draw(11,5) circle (3pt);
\draw(12,5) circle (3pt);
\draw(13,5) circle (3pt);
\draw(14,5) circle (3pt);
\draw(12,6) circle (3pt);
\filldraw(12,7) circle (3pt);
\draw(10.1,5)--(10.9,5);
\draw(11.1,5)--(11.9,5);
\draw(12.1,5)--(12.9,5);
\draw(13.1,5)--(13.9,5);
\draw(12,5.1)--(12,5.9);
\draw(12,6.1)--(12,6.9);
\draw(10,4.6) node {\footnotesize{$\alpha_6$}};
\draw(11,4.6) node {\footnotesize{$\alpha_5$}};
\draw(12,4.6) node {\footnotesize{$\alpha_4$}};
\draw(13,4.6) node {\footnotesize{$\alpha_3$}};
\draw(14,4.6) node {\footnotesize{$\alpha_1$}};
\draw(11.5,6) node {\footnotesize{$\alpha_2$}};
\draw(11.5,7) node {\footnotesize{$\alpha_0$}};

\draw(9,3) node {$F_4$};
\filldraw(10,3) circle (3pt);
\draw(11,3) circle (3pt);
\draw(12,3) circle (3pt);
\draw(13,3) circle (3pt);
\draw(14,3) circle (3pt);
\draw(10.1,3)--(10.9,3);
\draw(11.1,3)--(11.9,3);
\draw[double](12.1,3)--(12.9,3);
\draw(12.8,2.9)--(12.9,3);
\draw(12.8,3.1)--(12.9,3);
\draw(13.1,3)--(13.9,3);
\draw(10,2.6) node {\footnotesize{$\alpha_0$}};
\draw(11,2.6) node {\footnotesize{$\alpha_1$}};
\draw(12,2.6) node {\footnotesize{$\alpha_2$}};
\draw(13,2.6) node {\footnotesize{$\alpha_3$}};
\draw(14,2.6) node {\footnotesize{$\alpha_4$}};

\draw(9,1) node {$G_2$};
\filldraw(10,1) circle (3pt);
\draw(11,1) circle (3pt);
\draw(12,1) circle (3pt);
\draw(10.1,1)--(10.9,1);
\draw(11.1,1)--(11.9,1);
\draw(11.1,1.06)--(11.85,1.06);
\draw(11.1,0.94)--(11.85,0.94);
\draw(11.8,0.9)--(11.9,1);
\draw(11.8,1.1)--(11.9,1);
\draw(10,0.6) node {\footnotesize{$\alpha_0$}};
\draw(11,0.6) node {\footnotesize{$\alpha_1$}};
\draw(12,0.6) node {\footnotesize{$\alpha_2$}};

\end{tikzpicture}
\caption{Extended Dynkin diagrams, with the lowest root $\alpha_0 $ coloured.}
\end{figure}
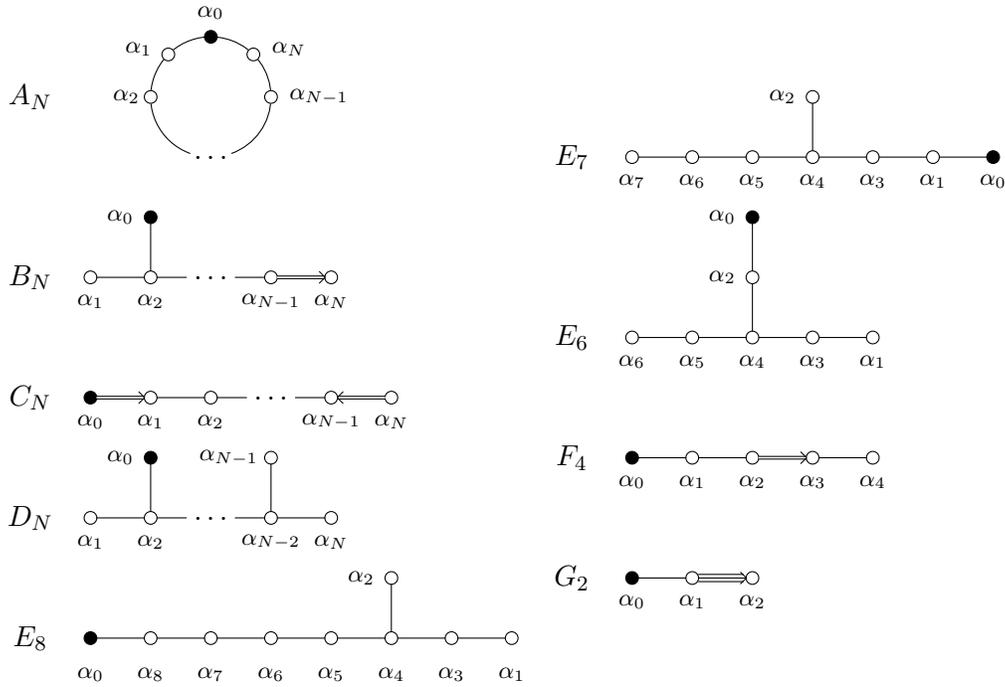

For a root system of type $ A_N $, the simple root coefficients $ m_j = 1 $ for all $ j = 1, \ldots , N $. Thus any diagram involution fixing some node is induced by an involution of the Lie algebra.  By inspection of the extended Dynkin diagram shown in Figure~\ref {figure:extended},
we see that when $ N $ is even, 
 every involution of the extended Dynkin diagram fixes some $\alpha_j $.
% Acting nontrivially on $\alpha_0 $ are reflections fixing some simple root. 
When $ N $ is odd we need to consider the rotation  $\pi (j) = j + \frac {1} {2} (N + 1)\bmod(N +1) $ and reflections.

For the involution  $\pi (j) = j + \frac {1} {2} (N + 1)\bmod(N +1) $, the root $\gamma=\alpha_1 + \alpha_2 +\ldots + \alpha_{\frac12(N-1)}$ satisfies conditions (a), (b), (c) above.

Consider now an involution $\pi $ coming from a reflection. Since we have automatically covered the cases when there is a fixed root we can assume that there are an even number of roots between $\alpha_0$ and $\pi(\alpha_0)$ going in each direction around the circle. Indeed the axis of reflection is between the nodes $(\pi(0)-1)/2$ and $(\pi(0)+1)/2$ and between $(N+\pi(0))/2$ and $(N+\pi(0))/2 + 1$. The roots
$$\gamma = \alpha_1 + \alpha_2 + \ldots + \alpha_{(\pi(0)-1)/2}\quad \text{and} \quad \delta = \alpha_{\pi(0)+1} + \alpha_{\pi(0)+ 2} \ldots +\alpha_{(\pi(0)+N)/2}$$
satisfy conditions (i), (ii), (iii) above.

There is only one involution of the root system of type $B_N$,  which sends $\alpha_0$ to $\alpha_1$ and fixes everything else. We can choose $\gamma = \alpha_2 + \ldots + \alpha_N$. 

For root systems of type $C_N$ there is again only one involution; $\pi(\alpha_i) = \alpha_{N-i}$. Here choose $\gamma= \alpha_1+\ldots +\alpha_{N-1}$.

For $ D_N $, $ m_1 = m_{N -1} = m_N = 1 $, and so we need only consider involutions which do not fix any of these vertices, of which there are three. These are involutions with $\pi(0) = 1, N-1$ or $N$. If $\pi(0)=1$ then let $\gamma = \alpha_2 + \ldots +\alpha_{N-1}$, and if $\pi(0)=N-1$ or $N$, take $\gamma =\alpha_1+\alpha_2 + \ldots + \alpha_{N-2 }$.

For the root system $ E_6 $, all involutions of the diagram fix the vertex $\alpha_4 $ and $ m_4 = 3 $ is odd.

The unique involution of the extended Dynkin diagram for $ E_7 $ satisfies $\pi (\alpha_0) =\alpha_7 $. A list of all positive roots of $ E_7 $ are tabulated for example in  \cite [p 1524-1530] {Vavilov:01}. 
Let $\gamma=\alpha_1 + \alpha_2 + 2\alpha_3 + 2\alpha_4 + \alpha_5 + \alpha_6$, so $\pi(\gamma)=\alpha_1 + \alpha_2 + \alpha_3 + 2\alpha_4 + 2\alpha_5 +\alpha_6$ and $\pi(\gamma)+\alpha_{\pi(0)}=\alpha_1 + \alpha_2 + \alpha_3 + 2\alpha_4 + 2\alpha_5 +\alpha_6+\alpha_7$ is also a root. Furthermore $\gamma + \alpha_{\pi(0)} + \pi(\gamma) = 2\alpha_1 +2\alpha_2+ 3\alpha_3 + 4\alpha_4 + 3\alpha_5 +2\alpha_6 +\alpha_7$ which is the highest root.

The extended Dynkin diagrams of type $ E_8, F_4, G_2 $ do not possess any involutions.

 We have then shown that given any involution $\pi $ of an extended Dynkin diagram for $ (\g ^\C,\t ^\C )$, there exists an involution $\Theta $ of $\g ^\C $ preserving $\t ^\C $ and inducing $\pi $. It remains to show that there is a real form $\g$ of $\g ^\C $ for which $\Theta $ is a Cartan involution and such that $\g\cap\t ^\C $ has full rank. For any choice of simple roots we may consider the corresponding {\em Borel subalgebra} $\mathfrak b ^\C = \mathfrak t ^\C \oplus \bigoplus_{\alpha\in\Delta^ +} \mathcal G ^\alpha $ and it is easy to see that $\Theta $ preserves the set of simple roots if and only if it preserves the corresponding Cartan and Borel subalgebras. Now by \cite [Theorem 8.6] {Kac:90} % Clinton please enter http://www.cambridge.org/gb/knowledge/isbn/item1145937/?site_locale=en_GB into my bibliography is a book
there exists an automorphism $\Psi $ of $\g ^\C $ such that $\Psi\Theta\Psi^ {- 1} $ acts on the corresponding simple and lowest root vectors $R_{\alpha_j} $ in the Chevalley basis simply by scaling them by $\pm 1 $, and hence preserves the Cartan and Borel subalgebras $\t ^\C $ and $\mathfrak b ^\C $. Then $\Theta $ preserves the Cartan subalgebra  $\Psi^ {- 1} (\t ^\C) $ and the Borel subalgebra $\Psi^ {- 1} (\mathfrak b ^\C) $ and hence the set of simple roots $\Psi^ {- 1}\{\alpha_1, \ldots ,\alpha_N\} $. 
% In particular then \todo{insert proof above of this, quote Knapp} the Cartan subalgebra $\Psi^{-1}(\t ^\C) $ is maximally compact and so we may apply 
Then  there exists a real form $\g' $ of $\g ^\C $ with respect to which $\Theta $ is a Cartan involution and  such that $\g'\Psi^{-1}(\t ^\C) $ is a Cartan subalgebra of $\g' $ \cite[proof of Theorem 6.88]{Knapp:02}.

 Let $\mathfrak{l} ^\C $ denote the $ (+1)$-eigenspace of $\Theta $ and $ L ^\C $ a complex Lie group with Lie algebra $\mathfrak{l} ^\C $. In  \cite [Theorem 1] {Matsuki:79} (c.f. \cite [Proposition 2.1] {Vogan:83}) it was shown that for a given real form $\g' $ and $\Theta $-stable Cartan subalgebra $\t ^\C $ of a simple complex Lie algebra $\g ^\C $, there exists a $\Theta $-stable Cartan subalgebra $\t' $ of $\g' $ and $ l\in L ^\C $ such that $\t ^\C =\Ad_l (\t') ^\C$. Hence $\g =\Ad_l\g' $ is a real form of $\g ^\C $ for which $\t =\g\cap\t ^\C $ is a $\Theta $-stable real form of $\t ^\C $ and $\Theta $ is a Cartan involution of $\g $.

By Proposition ~\ref {prop:Coxeter} the Coxeter automorphism corresponding to the choice of simple roots $\alpha_1,\ldots,\alpha_N $ preserves the real form $\g $ and  in particular the Cartan subalgebra $\mathfrak t $.\end {proof}

\section{Toda frame}\label{Toda}

%!tex root=finitenew.tex
We now explore the relationship between cyclic primitive maps and the affine Toda field equations. Henceforth $ G $ shall denote a simple real Lie group, $ T $ a Cartan subgroup and $\alpha_1, \ldots ,\alpha_N $ simple roots such that the resulting Coxeter automorphism $\sigma $ preserves the real group $ G $. This Coxeter automorphism then gives $ G/T $ the structure of a $ k $-symmetric space, where $ k -1 $ is the maximum height of a root of $\g ^\C $. We shall consider cyclic primitive maps $\psi $ from the complex plane into $ G/T $ and will see that % for $\Lambda\subset\C $ a real lattice, 
cyclic primitive maps $\psi:\C\rightarrow G/T $
 arise from and give rise to solutions of the
two-dimensional affine Toda field equations for $\g $. Our results also apply to maps from a simply-connected coordinate neighbourhood of any Riemann surface.

The
famous Toda equations arose originally as a model for particle interactions
within a one-dimensional crystal, with the affine model corresponding to the
particles being arranged in a circle. They have been the subject of extensive
study, both as a completely integrable Hamiltonian system and in the context
of Toda field theories.
% We will consider in this section the affine Toda field equation  for a simple Lie algebra $\g $
%and corresponding Toda frames of primitive maps $\xi:\C\rightarrow G/T $, where $ T $ is a Cartan subgroup and $ G/T $ the $ k $-symmetric space is given by the Coxeter automorphism $\sigma $ determined by a choice of simple roots $\alpha_1, \ldots ,\alpha_N $.
The standard form of the affine Toda field equation for $\g $ on the complex plane is
% $\mathfrak{so} (2 n, 1)$ on $\C$ is
\begin{equation}\label{eq:Toda}
2\Omega_{z\bz}  = \sum_{j = 0}^ N m_j e^{2 \alpha_j (\Omega)}\alpha_j ^\sharp
\end{equation}
Here $\Omega : \C \rightarrow i \mathfrak{t}$ is a smooth map, the lowest root  $\alpha_0 $ is given by
\[
\alpha_0 = -
\sum_{j = 1}^N m_j \alpha_j,
\]
we set $m_0 = 1$ and $ R_{\alpha_j} $ are root vectors such that $ \alpha_j ^\sharp $ is the dual of $\alpha_j $ with respect to the Killing form. Using  \eqref {eq:simplereality},  since the Coxeter automorphism preserves the real form $\g $ there exists a permutation $ \pi $ of the roots $\alpha_0,\alpha_1,\ldots,\alpha_N $
such that
\begin {equation}\label {eq:permutation}
\overline { \alpha_j} = -\alpha_{\pi (j)}.
\end {equation}
We shall consider the generalisation of the affine Toda field 
equations obtained by allowing $m_0$, $m_1, \ldots, m_N $ to be any positive real
numbers such that
 $ m_{\pi (j)} =\overline {m_j} $ and $ R_{\alpha_j} $ to be any root vectors satisfying $\overline {R_{\alpha_j}} = R_{-\alpha_{\pi (j)}} $.

Given a cyclic element $W = \sum_{j = 0}^N r_j R_{\alpha_j}$ of
$\mathfrak{g}^{\sigma}_1$, we say that a lift $F : \C
\to G$ of $\psi : \C \rightarrow G / T$ is a \emph{Toda frame} with
respect to $W$ if there exists a smooth map $\Omega : \C \rightarrow i
\mathfrak{t}$ such that
\begin{equation}
  \label{eq:Todaframe} F^{- 1} F_z = \Omega_z + \Ad_{\exp \Omega} W.
\end{equation}

We call $\Omega$ an \emph{affine Toda field} with respect to $W$. The
motivation for this nomenclature is

\begin{lemma}
Fix a cyclic element $W = \sum_{j = 0}^N r_j R_{\alpha_j}$ of
$\mathfrak{g}^{\sigma}_1$ such that  $ m_{\pi (j)} =\overline {m_j} $ and $\overline {R_{\alpha_j}} = R_{-\alpha_{\pi (j)}} $.

The affine Toda field equation  \eqref {eq:Toda}  is the integrability condition for the existence
  of a Toda frame with respect to $ W $ where we take $m_j = r_j \overline{r_j}$ for $j
  = 0, \ldots, N$.
\end{lemma}

\begin{proof}
  Using $[R_{\alpha_j}, R_{- \alpha_l}] = 0$ whenever $j \neq l$, we can
  rewrite the Toda field equation \eqref {eq:Toda} as
  \begin{align*}
    2{\Omega}_{z\bar{z}} &
    =\sum_{j,l=0}^ Nr_j\overline{r_l}e^{{\alpha}_j({\Omega})}e^{{\alpha}_j({\Omega})}[R_{{\alpha}_j},R_{-{\alpha}_l}]\\
    & =[\sum_{j=0}^ N r_je^{{\alpha}_j({\Omega})}R_{{\alpha}_j},
    \sum_{l=0}^ N\overline{r_l}e^{{\alpha}_l({\Omega})}R_{-{\alpha}_l}].
  \end{align*}

  From equation  \eqref  {eq:roots}   we know $e^{\alpha_j (\Omega)} R_{\alpha_j} = \Ad_{\exp
  \Omega} R_{\alpha_j}$ and also
\[
e^{\alpha_l (\Omega)} R_{- \alpha_l} = e^{-
  \alpha_l (- \Omega)} R_{- \alpha_l} = \Ad_{\exp - \Omega} R_{- \alpha_l} .
\]
  If we set  $W : = \sum_{j = 0}^ N r_j R_{\alpha_j}$ with the normalisation is described in the lemma then since $\sum_{j = 0} ^ NR_jR_{\alpha_j} =\sum_{j = 0} ^ NR_j\overline {R_{-\alpha_{j}}} $, the Toda field
  equation becomes
  \[ 2 \Omega_{z \bar{z}} = [ \Ad_{\exp \Omega} W, \Ad_{\exp - \Omega}
     \overline{W}]. \]
  Now for any given $\Omega : \C \to i\mathfrak {t} $ the integrability condition for the existence
  of a Toda frame with respect to $ W $ is the Maurer-Cartan equation  \eqref {eq:MC}  for
\[
\varphi = (\Omega_z +
  \Ad_{\exp \Omega} W) dz + (\Omega_{\bar z} +
  \Ad_{\exp -\Omega} W) d\bar z.
\]
 Namely, this integrability condition is
  \begin{align*}
    0 & =(-{\Omega}_{\bar{z}}+{\Ad}_{\exp
    -(\Omega)}{\overline{W}})_z-({\Omega}_z+{\Ad}_{\exp
    {\Omega}}W)_{\bar{z}}\\
    &\qquad +[{\Omega}_z+{\Ad}_{\exp
    {\Omega}}W,-{\Omega}_{\bar{z}}+{\Ad}_{\exp -(\Omega)}{\overline{W}}]\\
    & =-2{\Omega}_{z\bar{z}}+[{\Ad}_{\exp {\Omega}}W,{\Ad}_{\exp
    -{\Omega}}{\overline{W}}],
  \end{align*}
which is precisely the Toda field equation.
\end{proof}

Recall that we write $
 \alpha_0 = - \sum_{j = 1}^ N m_j \alpha_j$ 
for the expression of the lowest root $\alpha_0 $ in terms of the chosen simple roots $\alpha_1,\ldots,\alpha_N $.
Given $\tilde{F} : \C \to G $ with
\begin {equation}
\label {eq:rootcoefficients}
\tilde F\inv \tilde F_z|_{\g_1 ^\sigma} = \sum_{j=0}^ N c_j R_{\alpha_j},
\end {equation}
we say that a cyclic element
\begin {equation}\label {eq:W}
 W=\sum_{j=0}^N r_j R_{\alpha_j} 
\end {equation}
   $\mathfrak{g} ^\sigma_1 $
 is \emph{normalised with respect to} $\tilde{F} : \C \to G $ if
\[
r_0 \prod_{j = 1}^ N r_j^{m_j} = c_0 \prod_{j = 1}^ N c_j^{m_j}.
\]
%\[
%c_1^2c_2^2\ldots c_{n-1}^2 c_n c_0 = r_1^2r_2^2\ldots r_{n-1}^2 r_n r_0 \text { and } r_k, r_0 \in \R^+.
%\]
\begin{theorem}\label {theorem:Toda}
A map $\psi:\C\rightarrow G/T $ possesses a Toda frame if and only if it is cyclic primitive
% Let $ G $ be a simple Lie group, $\sigma $ the  Coxeter automorphism determined by a choice of Cartan subgroup $ T $ and simple roots and
% and $G
%  / T$ the resulting $k$-symmetric space
  Let $\psi : \C \to G / T$ be a cyclic primitive map possessing a frame $\tilde{F} : \C \to G$  such that $c_0 \prod_{j =
  1}^ N c_j^{m_j}$ is constant, where $ c_j $ are the root coefficients defined in  \eqref  {eq:rootcoefficients}. % (where $c_j$ are defined by $\tilde F\inv \tilde F_z|_{\g_1 ^\sigma} = \sum_{j=0}^ N c_j R_{\alpha_j}$).
Then for any cyclic element $W $ of $\g ^\sigma_1 $ which is normalised with respect to $\tilde F $
  there exists a Toda frame $F : \C \to G$ of $\psi$ with respect to $W$. Furthermore if $\psi$ and $\tilde F$ are doubly periodic with lattice $\Lambda$  then so is the Toda frame $ F $.

Conversely, if $\psi:\C \to G/T$ has a Toda frame $ F $ with respect to cyclic $ W\in\g_1 ^\sigma $
then $\psi$ is cyclic primitive and $ W $ is normalised with respect to $ F $. In particular then the root coefficients $c_j$ are such that $c_0 \prod_{j =
  1}^ N c_j^{m_j} $ is constant. % =r_0 \prod_{j = 1}^ N r_j^{m_j} $, a non-zero constant.
  
\end{theorem}

\begin{proof}
  Consider the frames $F : = \tilde{F} \exp X$ of $\psi$ where $X : \C \to
  \mf{t}$. For such $F$ we have $F^{- 1} F_z = \Ad_{\exp - X} \tilde{F} \inv
  \tilde{F}_z + X_z$ and so
  \[ F^{- 1} F_z |_{\g_1 ^\sigma} = \Ad_{\exp - X} \tilde{F} \inv
     \tilde{F}_z |_{\g_1 ^\sigma} . \]
  This implies the Toda condition of $\Ad_{\exp \Omega} W = F^{- 1} F_z
  |_{\g_1 ^\sigma}$ is equivalent to
    \begin{align}
\Ad_{\exp(X+{\Omega})}W=\tilde{F}^{-1}\tilde{F}_z|_{\g_1 ^\sigma}=\sum_{j=0}^ Nc_jR_{{\alpha}_j}.
\end{align}
  Using equation \eqref{eq:roots} we can rewrite this as
  \begin{align*}\label{eq:Omega+X}
\sum_{j=0}^ Nr_je^{{\alpha}_j(X+{\Omega})}R_{{\alpha}_j}=\sum_{j=0}^ Nc_jR_{\alpha_j}.
\end{align*}

  Comparing root space coefficients implies that
\begin {equation}
\label {eq:Xplus}
e^{\alpha_j (X + \Omega)} =
  \frac{c_j}{r_j}\text { for $j = 1, \ldots k$}
\end {equation}
 and $r_0 \prod _{j = 1}^ N (e^{\alpha_j
  (X + \Omega)})^{- m_j} = c_0$. Since $ W $ is normalised with respect to $\tilde F $
% $r_0 \prod_{j = 1}^ N r_j^{m_j} = c_0 \prod_{j= 1}^ N c_j^{m_j}$
 and $\C$ is simply connected, we can solve for $X + \Omega$. We can then find $\Omega$
  and $X$ from $X + \Omega$ by taking its $\mf{t}$ and $i\mf{t}$ components
  respectively.

It remains to show that $\Omega_z dz = F \inv \partial F |_{\mf{t}} = \varphi'_{\mf{t}}$.
  From the $\g_1 ^\sigma$ component  \eqref {eq:MCg1} of the Maurer-Cartan equation for $\varphi $ we
  have
  \[ \partial ( \Ad_{\exp \Omega} W) - [ \Ad_{\exp \Omega} W,
     \varphi'_{\mf{t}}] = 0 \]
  or equivalently
  \[ [ \Ad_{\exp \Omega} W, \varphi'_{\mf{t}} - \partial \Omega] = 0. \]
  Since $W$ is cyclic so is $\Ad_{\exp \Omega} W$ and thus $\varphi'_{\mf{t}} =
  \partial \Omega$.

Conversely, given $W$  and a solution $\Omega$ to the corresponding affine Toda field equation, the resulting Toda  frame $F$ is primitive. Furthermore the equation
\[
 r_0 (e^{- \sum_{j = 1 } ^ N  m_j \alpha_j (X + \Omega)} R_{\alpha_0} + \sum_{j = 1}^ N r_j e^{\alpha_j
  (X + \Omega)} R_{\alpha_j} = \sum_{j = 0}^ N c_j R_{\alpha_j}
\]
 implies that
  $r_0 \prod_{j = 1}^ N r_j^{m_j} = c_0 \prod_{j = 1}^ N c_j^{m_j}$ and hence $c_0 \prod_{j = 1}^ N c_j^{m_j}$ is a non-zero constant. This implies that the $c_j$ are nowhere zero and $\psi$ is cyclic primitive.

%Suppose now that we know that $F: \C \to G$ is a Toda frame for $\psi:\C \to G/T$. By definition there is some $\Omega:\C \to i\t$ such that $F^{-1}F_z = \Omega_z + \Ad_{\exp \Omega} W \in \g_0 \oplus \g_1$. The corresponding coefficients $c_j$ satisfy 
%$$\sum_{j=1}^N c_j R_{\alpha_j} = \sum_{j=0}^N \Ad_{\exp \Omega} r_jR_{\alpha_j}$$
%and hence 
%$c_0 = r_0 e^{\alpha_0(\Omega)} = r_0 e^{-\sum_{j=1}^N m_j \alpha_j(\Omega)}$ and $c_j = r_j e^{\alpha_j(\Omega)}$ for $j=1, \ldots, N$. Combining these relations lets us obtain
%$c_0\Pi_{j=1}^N {c_j}^{m_j} = r_0 \Pi_{j=1}^N {r_j}^{m_j}$ which is a non-zero constant as the $r_j$ are all non-zero constants. Further this implies that the $c_j$ are nowhere zero and hence $\psi$ is cyclic primitive.
%

Now suppose $\tilde{F}$ is doubly periodic with respect to a lattice $\Lambda $. Then  for $j=1,\ldots N$, from\eqref {eq:Xplus} we see that $e^{\alpha_j(X+\Omega)}$ is  doubly periodic with respect to $\Lambda$ and so
\[
\exp(X + \Omega) = \exp(\sum_{j=1}^ N \alpha_j(X+\Omega)\eta_j)
\]
 is also.
Given any $\Gamma \in \Lambda$ it follows that % $\exp(X(z+\Gamma) + \Omega(z+\Gamma)) = \exp(X(z) + \Omega(z))$ for all $z\in \C$ and hence
\begin{align}\label{eq:conj1}
\exp(X(z+\Gamma) -X(z)) = \exp(\Omega(z)-\Omega(z+\Gamma)).
\end{align}
 Using the conjugation map $\g^\C \to \g^\C$ which fixes $\g$, we obtain from \eqref{eq:conj1} that
\begin{align}\label{eq:conj2}
\exp(X(z+\Gamma) -X(z)) = \exp(-\Omega(z)+\Omega(z+\Gamma).
\end{align}
When combined, \eqref{eq:conj1} and \eqref{eq:conj2} imply that $\exp(X(z+\Gamma))=\exp(z))$ for all $z$ and hence $\exp X$ is doubly periodic with lattice $\Lambda$.

Since $\tilde F$ and $\exp X$ are both  doubly periodic with lattice $\Lambda$ we know $F =\tilde F \exp X$ is also.
\end{proof}

Our chief interest lies in cyclic primitive $\psi $ which are doubly periodic, as it is these we shall show are of finite type. We henceforth restrict our attention to doubly-periodic maps and denote by $ \Tor $ any genus one Riemann surface.
 Let $W$ be a cyclic element of $\g _1 ^\sigma $ as before. We say that a frame $F : \Tor \to G$ of $\psi : \Tor \rightarrow G / T$ is a Toda frame with respect to $W$ if $F$ is a Toda frame of $\psi$ when both are considered as maps from $\C$. From the proof of Theorem~\ref {theorem:Toda} we make the following observation, which will prove useful in the next section.
\begin{lemma}\label {lemma:periodic}
If $F:\Tor \to G$ is a Toda frame of $\psi: \Tor \to G/T$ then the corresponding affine Toda field $\Omega : \C \rightarrow i\mathfrak{t}$ has the property that $\exp \Omega$  and  $\Omega_z$ are doubly periodic with lattice $\Lambda$.
\end{lemma}
%
% we call a Toda frame of a cyclic primitive lift $\psi $ a Toda frame of $ f $. From
%Theorems~\ref{thm:cyclicprimitive} and \ref {theorem:Toda}
% the following corollary is immediate.
%\begin {corollary} Let $ f: \C\rightarrow S ^ {2n}_1 $ be a superconformal harmonic map such that the elements $ f_1, \ldots, f_n$ of the harmonic sequence of $ f $ are everywhere defined. Let $ W =\sum_{j = 1} ^ nr jR_{\alpha_j} $ is a cyclic element of $\g_1 ^\sigma $
%normalised with respect to the cyclic primitive frame of $ f $ defined by its harmonic sequence as in Theorem~\ref{theorem:Toda}. Then $ f $ has a Toda frame with respect to $ W $.
%\end {corollary}
%
%Given a superconformal harmonic map $ f:\Tor\rightarrow S ^ {2n}_1 $,
%$ W =\sum_{j = 1} ^ nr jR_{\alpha_j} $ is a cyclic element of $\g_1 ^\sigma $
% is normalised with respect to $ f $ if Recall from Theorem~\ref{thm:cyclicprimitive} that $\langle \partial^n f, \partial^n f\rangle =2^{2n}c_1^2 c_2^2 \ldots c_{n-1}^2 c_n c_0$ is then a non-zero constant, where the $ c_j $ are the root vector coefficients defined in \eqref {eq:matrix}. In this case any cyclic $ W\in\g_1  ^\sigma  %$ %%an be normalised by multiplication by a complex number.
%  and Corollary\ref~{cor:primitive}
%we have the following corollary.

\section{Finite type result}\label{finite}
%Pretty much an edited version of 6.2 (Jacobi Fields) and 6.3 (proof that superconformal...)

%!TeX root=finitenew.tex

We will now show that all smooth maps $ \psi $ from a 2-torus $\Tor$ into the $ k $-symmetric space $G/T$ which have a Toda frame are of finite type. Hence all such maps can be constructed from a pair of commuting ordinary differential equations on a finite-dimensional loop algebra. In \cite {BFPP:93} it was  shown that all semisimple adapted harmonic maps of a 2-torus into a compact semisimple Lie group are of finite type. We prove our finite type result by adapting the methods of that paper. Note that the existence of a Toda frame forces $\psi $ to be cyclic primitive.

% Recall in particular that $ G $ is any semisimple Lie group, not necessarily compact.

A map $Y : \Tor  \to \mf{g}^\C$ is called a \emph{Jacobi field} if there exists $\dot\Omega: \Tor \to \mf{t}^\C$ such that
\begin{align}\label{eq:Jacobi}
dY + [F\inv dF, Y] =
\left(\dot{\Omega}_z + [\dot\Omega, F\inv F_z ]\right)dz
+\left(-\dot\Omega_{\bz} - [\dot\Omega,F\inv F_{\bz}]\right)d\bz.
\end{align}
If $F_t$ is a family of Toda frames with corresponding $\Omega_t:\C \to i\mf{t}$ then $\frac{d}{dt} F_t |_{t=0}$ is a Jacobi field with $\dot{\Omega} = \frac{d}{dt}\Omega_t|_{t=0}$. Note that if $\dot{\Omega}=0$ the Jacobi equation is the Killing field equation.

Let $F$ be a Toda frame for  $\psi:\C\rightarrow G/T $. We have
\begin{align*}
F\inv dF =
 (\Omega_z +\Ad_{\exp \Omega} W )dz + (-\Omega_{\bz} + \Ad_{\exp -\Omega}\overline{W})d\bz
\end{align*}
for some $\Omega: T ^ 2\to i\mf{t}$ and cyclic $ W\in\g_1 ^\sigma $. Let $Y$ be a Jacobi field with corresponding $\dot{\Omega}: T ^ 2\to i\mf{t}$.
Then $ Y $ must satisfy
\begin{align}
Y_z + [\Omega_z +\Ad_{\exp \Omega} W, Y]&= \dot{\Omega}_z + [\dot{\Omega},\Ad_{\exp \Omega} W]\label {eq:Jacobi1}\\
Y_{\bz} + [-\Omega_{\bz} + \Ad_{\exp -\Omega}\overline{W}, Y]&= -\dot{\Omega}_{\bz} - [\dot{\Omega},\Ad_{\exp -\Omega}\overline{W}].\label {eq:Jacobi2}
\end{align}
Taking   \eqref{eq:Jacobi1}$_{\bz} - $ \eqref{eq:Jacobi2}$_z$ we obtain
\begin{equation*}
2\dot{\Omega}_{z\bar{z}}=-\bigl[\Ad_{\exp \Omega} W,[\dot{\Omega},\Ad_{\exp -\Omega}\overline{W}]\bigr] -\bigl[\Ad_{\exp -\Omega}\overline{W},[\dot{\Omega},\Ad_{\exp \Omega} W]\bigr].
\end{equation*}

Since $\Omega$ and $W$ are fixed, we see that $\dot{\Omega}$ satisfies a linear elliptic partial differential equation. As the torus is compact, the space of possible $\dot{\Omega}$ is finite dimensional.

\begin{lemma}\label{lem:Jacobi}
% Let $ G $ be a semisimple Lie group, $ T $ a Cartan subgroup and $G/T$ be the $k$-symmetric space induced by the Coxeter automorphism $\sigma $.
Suppose $\psi: \Tor\rightarrow G/T$ is a cyclic primitive map possessing % has a Toda frame $ F:\Tor \to G$
  a formal Killing field $Y=\sum_{j\leq 1} \lambda^j Y_j\in \Omega^{\sigma}\mf{g} ^\C$.  Then $\psi$ has a (real) polynomial Killing field with highest term $Y_1$.
%satisfying
%\[
%dY = [Y,\varphi _\lambda],
%\]
%where $\varphi_\lambda $ is the loop of 1-forms
%\[
%\varphi_\lambda = (\Omega_z +\lambda\Ad_{\exp \Omega} W )dz + (-\Omega_{\bz} +\lambda^ {- 1} \Ad_{\exp -\Omega}\overline{W})d\bz
%\]
%defined in
%\eqref{eq:primitiveflat}.
\end{lemma}

\begin{proof} % with a Toda Frame $F:\Tor \to G$
We will find an infinite number of linearly independent Jacobi fields for which some linear combination must be a formal Killing field.
%\[
%d\xi = [\xi,\varphi_\lambda] %
%\]
Since $Y$ is a formal Killing field, we have \eqref{eq:lax}.
\[
\sum_{j\leq 1} \lambda^j dY_j
= \left[\sum_{j\leq 1} \lambda^j Y_j,\varphi_\lambda \right].
\]
Comparing coefficients of $\lambda^j$ gives the equations
\begin{align*}
 (Y_{j})_zdz + [\varphi'_{\mf{t}} , Y_j] + [\varphi'_{\mf{p}}, Y_{j-1}] &= 0, \\
 (Y_{j})_{\bz}d\bar{z} + [\varphi''_{\mf{t}}, Y_j] + [\varphi''_{\mf{p}}, Y_{j+1}] &= 0.
\end{align*}
For each $l\in \Z^+$ set
\[
Y^{l} := \frac12 Y_{-kl} + \sum_{-kl <j \leq 1} \lambda^{j+kl} Y_j.
\]
We will show that the $Y^l$ are all Jacobi fields.  Considering the coefficients separately gives
\begin{align*}
(Y^l)_zdz + [ \lambda \varphi'_{\mf{p}} + \varphi'_{\mf{t}} , Y^l]
&= \frac12 (Y_{-kl})_{z}dz + \left[\frac12 Y_{-kl}, \lambda \varphi'_{\mf{p}}\right]\\
(Y^l)_\bz d\bar{z}+ \left[\varphi''_{\mf{t}} + \lambda^{-1} \varphi''_{\mf{p}},  Y^l\right]
&= -\frac12 (Y_{-kl})_{\bz}d\bar{z} - \left[\frac12  Y_{-kl}, \lambda^{-1} \varphi''_{\mf{p}}\right].
\end{align*}
Since $Y_{-kl} \in \mf{g}_0 = \mf{t}^\C$ we can set $\dot\Omega^l := \frac12 Y_{-kl}$. With this choice of $\dot{\Omega}$, $Y$ is a solution to \eqref{eq:Jacobi} and hence is a Jacobi field. The space of potential $\dot\Omega$ is finite dimensional, so there must be a non-trivial finite linear combination of the $\dot\Omega^l$ which equals $0$. The corresponding finite linear combination  of the $Y^l$ is a formal Killing field. Since the highest order terms of the $Y^{l}$ are each $Y_1$ we can rescale this formal Killing field to one with highest order term $Y_1$.
After multiplying by an appropriate power of $\lambda^k$ we may also assume that the degree of the lowest term has smaller absolute value than the degree of the highest term. % Smallest $l$ such that $\xi_l \neq 0$ has <  d-1 $.
 Then  $ \overline{\xi} +\xi $
% lies in some real finite dimensional twisted  loop algebra $\A maker ^\sigma_Delta $ and satisfies the Lax equation
is a polynomial Killing field for $\xi$ and by construction has highest order term $ Y_1 $.
\end{proof}

\begin{theorem} % [c.f. \cite {BFPP:93}, Theorem 7.1]
\label{thm:finite}
% Let $ G $ be a semisimple Lie group, $ T $ a Cartan subgroup and $G/T$ be the $k$-symmetric space induced by the Coxeter automorphism $\sigma $.
Suppose $\psi: \Tor\rightarrow G/T$ has a Toda frame $ F:\Tor \to G$. Then $\psi $ is of finite type.
\end{theorem}
\begin{proof}
 Let $F:\Tor \to G$ be the Toda frame of $\psi$ with corresponding $\Omega: \Tor \to i\mf{t}$ and $ W\in\g_1 ^\sigma $.
 Recall that $\psi$ is of finite type if it has an adapted polynomial Killing field $\xi$, that is a  $\xi= \sum_{j=-d}^d \lambda^j \xi_j $ in the real twisted loop algebra $\Omega^\sigma \g$ satisfying the Killing field equation \eqref{eq:lax} and such that
\[
\xi_d + \lambda\frac 12 \Ad_{\exp \Omega}W.
\]
Since $ G $ was assumed simple, the complexified Lie algebra $\g ^\C $ is simple and hence has a faithful linear representation so can be regarded as a subalgebra of some $\mathfrak{gl} (m,\C) $. If we set
\[
D = d - \ad_{\Omega_z dz- \Omega_{\bar{z}}d\bar{z}}
\]
then we can rewrite \eqref{eq:lax} as
\[
D\xi_\lambda = [\xi_\lambda, (2\Omega_z + \lambda \Ad_{\exp \Omega} W)dz + (-2\Omega_{\bar{z}} + \lambda^{-1}\Ad_{\exp -\Omega} \overline{W})d\bar{z}].
\]
From $d(\Ad_{\exp \Omega} W) = [\Omega_z, \Ad_{\exp \Omega} W]dz +[\Ad_{\exp \Omega} W, \Omega_{\bar{z}}]d\bar{z} $ we know $D\Ad_{\exp \Omega} W =0$.

Writing $V=\ker \ad_{\Ad_{\exp \Omega}W}$ and $V^{\perp} = \im \ad_{\Ad_{\exp \Omega}W}$, we have a bundle decomposition $\Tor \times \g ^\C = V \oplus V^{\perp}$. Furthermore
\[
VV \subset V, \quad V^{\perp}V \subset V^{\perp}, \quad V V^{\perp} \subset V^\perp.
\]

Let $X=\sum_{k\leq -1} \lambda^{k}X_k$ where the $X_k$ are sections of $V^\perp$. We seek $X$ such that
\begin {equation}\label{eq:Y}
Y=(1+X)^{-1}\Ad_{\exp \Omega}W(1+X)
\end {equation}
is a solution of the Killing field equation. Note that
\begin{align*}
DY=(1+X)^{-1}[\Ad_{\exp \Omega}W, DX(1+X)^{-1}](1+X),
\end{align*}
and define a one-form $\kappa$ by
$$\kappa =(1+X)((2\Omega_z + \lambda \Ad_{\exp \Omega}W)dz + (-2\Omega_{\bar{z}} + \lambda^{-1}\Ad_{\exp -\Omega}\overline{W})d\bar{z}-DX)(1+X)^ {- 1}.$$
Routine calculations show that
\begin{align*}
DY + [(2\Omega_z + \lambda \Ad_{\exp \Omega}W)dz &+ (-2\Omega_{\bar{z}} + \lambda^{-1}\Ad_{\exp -\Omega}\overline{W})d\bar{z}, Y]\\
& = (1+X)\inv [\Ad_{\exp _\Omega} W, -\kappa](1+ X)
\end{align*}
and hence $Y$ satisfies the Killing field equation if and only if  $\kappa$ takes values in $V$.

Our task then is to construct $X$ so that $\kappa$  takes values in $V$. We have
% Taking a nowhere vanishing holomorphic vector field $\dfrac {\partial} {\partial z} $ on $ \Tor $ we have that % Firstly we will consider only the $dz$ part. The $Z$ % denote the vector field in the $z$ direction.
\[
\kappa'\cdot  (1+X) = (1+X)(2\Omega_z + \lambda \Ad_{\exp \Omega}W) d z-\partial X
\]
where $\kappa'\cdot (1+ X) $ denotes multiplication. % , not the action of $\kappa ' $ on $ 1+ X $.
Note that $\Omega_z $ is valued in $ V^\perp$ as it lies in $\t^\C$.

The splitting of $\kappa' \cdot (1+X)$ into its $V$ and $V^\perp$ components is
\begin{align*}
(V): & \quad \kappa' = \left (\lambda \Ad_{\exp \Omega}W + (2X\Omega_z)^V\right)dz\\
(V^\perp):& \quad \kappa' \cdot X =\left (2\Omega_z +(2X\Omega_z)^\perp + \lambda X\Ad_{\exp \Omega}W\right)dz -D'X.
\end{align*}
Substitution implies
\begin{align}\label{eq:Vcomp}
\lambda[\Ad_{\exp \Omega}W,X] dz = 2 \Bigl(\Omega_z   +(X\Omega_z)^\perp   -(X\Omega_z)^V X\Bigr) dz -D' X.
\end{align}
Conversely if \eqref{eq:Vcomp} holds then $\kappa' =\left ( \lambda \Ad_{\exp \Omega}W + (2X\Omega_z)^V\right) dz $ and so $\kappa' $ takes values in $V$.
Comparing the $\lambda^j$ coefficients on both sides of \eqref{eq:Vcomp} we can solve for $X$ inductively over $j$ by at each stage requiring $X_j \in  \im \ad_{\Ad_{\exp \Omega}W}$ and
\begin{align*}
[\Ad_{\exp \Omega}W,X_1] &= 2\Omega_z\\
[\Ad_{\exp \Omega}W, X_{j-1}] dz &= 2 \Bigl ((X_j \Omega_z)^\perp    -\sum_{s+l=j}(X_s\Omega_z)^V X_l\Bigr) dz -D X_k.
\end{align*}
Define $\D = d + \ad_{\varphi_\lambda} $ and note that \eqref{eq:flat} says precisely that $\D$ is a flat connection in the trivial bundle $\Tor\times\g ^\C $.
% D + \ad_{(2\Omega_z +\lambda \Ad_{\exp \Omega}W)dz + (-\Omega_{\bar{z}} + \lambda^{-1}\Ad_{\exp -\Omega}\overline{W})d\bar{z}}$. The flatness of $\D % $ ensures that $\D_Z \D_{\bar{Z}} = \D_{\bar{Z}}\D_Z$. We can rewrite the Killing field equation now as $\D \xi_{\lambda}=0$.
With $ X $ as above we have $\D' Y=0$. We wish to show that $\D'' Y=0$ also, as this will imply that $Y$ satisfies the Killing field equation \eqref {eq:lax}.

Define $B$ by
\begin{equation}\label{eq:B}
\D'' Y = (1+X)^{-1}B (1+X).
\end{equation}
Using $\Ad_{\exp \Omega}W = (1+X) Y(1+X)^{-1}$  and
\[
\D'' \Ad_{\exp \Omega}W = [-\Omega_{\bar{z}} + \lambda\inv \Ad_{\exp -\Omega}\overline{W},  \Ad_{\exp \Omega}W]d\bar z
\]
we obtain
\begin{align*}
B d\bar z &=[(-\Omega_{\bar{z}} + \lambda^{-1}\Ad_{\exp -\Omega}\overline{W}) d\bar z - \D'' X (1+X)^{-1},\Ad_{\exp \Omega}W ]
\end{align*}
which shows that $B$ takes values in $V^\perp$.

% Define $\delta $ by $\D" Y = (1+X)^{-1}\delta (1+X) d\bar z $.
As $\D$ is a flat connection we have commutativity of covariant derivatives and hence $\D'\D''Y=0$ which we write as
\begin{align}\label{eq:DZB1}
-\D' X (1+X)^{-1}B  + \D' B  + B \D' X (1+X)^{-1} = 0.
\end{align}
Since $\D' B = D' B  + [2\Omega_z +\lambda \Ad_{\exp \Omega}W,B] dz $, we can rewrite \eqref{eq:DZB1} as
\begin{align}\label{eq:DZB2}
D' B&=[\D' X (1+X)^{-1} + (\lambda \Ad_{\exp \Omega}W-2\Omega_z ) dz, B].
\end{align}
From its defining equation
\eqref {eq:B} we know that $ B $ is of the form $\sum_{j\leq d}\lambda^{j} B_j$. We will show that $ B =0$. Suppose not, then there is some non-zero top coefficient $B_d$. Since $X$ has only negative powers of $\lambda$, the $\lambda^{d+1}$ term in \eqref{eq:DZB2} is
\[
[\Ad_{\exp \Omega}W, B_d].
\]
 However we know that $B_d \in V^\perp$ and hence it must be zero.
Thus $\D'' Y = 0$ and $Y$ satisfies the Killing field equation. From \eqref{eq:Y} we see that $Y$ is of the form $\sum_{j\leq 0}\lambda^{j}Y_j$ and
furthermore $Y_0=\Ad_{\exp \Omega}W$.

We now need to project this $Y$ onto $\Omega^\sigma(\g^\C)$ to get a solution to the Killing field equation in the correct loop algebra.

 Representations of simple Lie algebras are completely reducible and we have identified $\g^\C$ with a subalgebra of $\mf{gl}(m,\C)$ so it must have a complementary subspace in $\mf{gl}(m,\C)$ which is invariant under the adjoint action of $\g^\C$. This means there exists a projection map $\pi: \Omega (\mf{gl}(m,\C)) \to \Omega(\g ^\C)$ such that
\[
d \pi(Y) = \pi(d Y) = \pi([Y, \varphi_\lambda]) = [\pi(Y), \varphi_\lambda].\]
Thus we have that $\pi(Y)\in \Omega(\g ^\C)$ satisfies the Killing field equation. Furthermore  $Y_0 = \pi(\Ad_{\exp \Omega}W) =\Ad_{\exp \Omega}W$. Set $\tilde{Y} = \lambda Y = \sum_{j \leq 1}\lambda^j Y_{j-1}$ and note that $\tilde{Y}_1 = Y_0 = \Ad_{\exp \Omega}W$.

We want to project $\tilde Y$ onto $\Omega^\sigma(\g^\C)$. Consider the map
$$\pi^{\sigma}_j := \frac 1 {k} (\Id + \epsilon^{-j} \sigma^j +  \epsilon^{-2j} \sigma^{2j} + \ldots +\epsilon^{-(k-1)j}\sigma^{(k-1)j})$$ where $\epsilon$ is the $k$-th primitive root of unity.  This map $\pi_j^\sigma$ projects any element in $\g^\C$ to its part in $\g_j$.
Thus we can define $\pi^\sigma : \Omega(\g ^\C)  \to \Omega^\sigma (\g^\C)$ by
\[
\pi^\sigma (\sum_j \lambda^j \xi_j) = \sum_j \lambda^j \pi_j^\sigma (\xi_j).
\]
 %\footnote{Note that this map is a correction of the ``averaging map'' in [BPW], defined on page~135. Their map projected everything to the eigenspace $\mf{g}_0$.}
Then $\tilde{\xi} = \pi^\sigma (\tilde{Y})$ satisfies
\begin{align*}
d\tilde{\xi}=[\tilde{\xi},\Omega_z + \lambda \Ad_{\exp \Omega} W) dz +  ( - \Omega_{\bz} +\lambda\inv\Ad_{\exp -\Omega}\overline{W}) d\bar{z}]
\end{align*}
and $\tilde{\xi}_1=\tilde{Y}_1=\Ad_{\exp \Omega} W$.

Now we may apply Lemma~\ref{lem:Jacobi} to $\tilde{\xi}$ to conclude the existence of a (real) polynomial Killing field $\xi$ whose top term, $\xi_d$, is $\Ad_{\exp \Omega} W$.

The $d-1$ coefficient of $\xi_z =[\xi,  \Omega_z + \lambda\Ad_{\exp \Omega}W]$
 is
$$(\Ad_{\exp \Omega}W)_z= \left[\Ad_{\exp \Omega}W,  \Omega_z\right] + [\xi_{d-1}, \Ad_{\exp \Omega}W] $$
which implies
$$\left[\xi_{d-1}-2\Omega_z,\Ad_{\exp \Omega}W\right]= 0.$$ Since $W$ is a cyclic element and  $\xi_{d-1}-2\Omega_z\in \t$ we conclude $\xi_{d-1}-2\Omega_z=0$ and hence $\xi$ satisfies the theorem.
\end{proof}
%From Theorem~\ref {thm:cyclicprimitive}, Corollary~\ref{cor:primitive} and Theorem~\ref{thm:finite} we have
%
%\begin{corollary}\label{cor:finite}
%Let $ f:\Tor\rightarrow S ^ {2n}_1 $ be a superconformal harmonic map with globally defined harmonic sequence $\{f_1, \ldots , f_n\} $. Then $ f $ has a lift $\psi:\Tor\rightarrow SO (2n, 1)/T $ of finite type.
%\end{corollary} 

% \section {Applications to Willmore  surfaces  in $ S ^ 3 $}
%\input{Willmore}\label{willmore}
\bibliographystyle{plain}
\def\cprime{$'$}

%\bibliography{C:\papers\harmonic}

\begin{thebibliography}{10}

\bibitem{BD:81}
A.~A. Belavin and V.~G. Drinfel{\cprime}d.
\newblock Solutions of the classical {Y}ang-{B}axter equation for simple {L}ie
  algebras.
\newblock {\em Funktsional. Anal. i Prilozhen.}, 16(3):1--29, 96, 1982.

\bibitem{Bobenko:91}
A.I. Bobenko.
\newblock All constant mean curvature tori in $\bf{R^3}$, ${S}^3$ and ${H}^3$
  in terms of theta-functions.
\newblock {\em Math. Ann.}, 290(2):209--245, 1991.

\bibitem{BPW:95}
J.~Bolton, F.~Pedit, and L.M. Woodward.
\newblock Minimal surfaces and the affine {T}oda field model.
\newblock {\em J. Reine. Angew. Math.}, 459:119--150, 1995.

\bibitem{BFPP:93}
F.~Burstall, D.~Ferus, F.~Pedit, and U.~Pinkall.
\newblock Harmonic tori in symmetric spaces and commuting {H}amiltonian systems
  on loop algebras.
\newblock {\em Annals of Math.}, 138:173--212, 1993.

\bibitem{BP:94}
F.~E. Burstall and F.~Pedit.
\newblock Harmonic maps via {A}dler-{K}ostant-{S}ymes theory.
\newblock In {\em Harmonic maps and integrable systems}, Aspects Math., E23,
  pages 221--272. Vieweg, Braunschweig, 1994.

\bibitem{Burstall:95}
F.E. Burstall.
\newblock {Harmonic tori in spheres and complex projective spaces.}
\newblock {\em J. Reine Angew. Math.}, 469:149--177, 1995.

\bibitem{CT:12}
E.~Carberry and K.~Turner.
\newblock Harmonic tori in de {S}itter spheres.
\newblock In preparation.

\bibitem{CE:75}
J.~Cheeger and D.G. Ebin.
\newblock {\em Comparison Theorems in Riemannian Geometry}.
\newblock Lecture Notes in Math. American Mathematical Society, 1975.

\bibitem{FPPS:92}
D.~Ferus, F.~Pedit, U.~Pinkall, and I.~Sterling.
\newblock Minimal tori in ${S}^4$.
\newblock {\em J. Reine. Angew. Math.}, 429:1--47, 1992.

\bibitem{Hitchin:90}
N.~Hitchin.
\newblock Harmonic maps from a 2-torus to the 3-sphere.
\newblock {\em J. Differential Geom.}, 31:627--710, 1990.

\bibitem{Kac:90}
V.~Kac.
\newblock {\em Infinite dimensional {L}ie algebras}.
\newblock Cambridge University Press, 1994.

\bibitem{Knapp:02}
W.~Knapp.
\newblock {\em {L}ie groups beyond an introduction}.
\newblock Birkh\"auser, Boston, MA, 2002.

\bibitem{Matsuki:79}
Toshihiko Matsuki.
\newblock The orbits of affine symmetric spaces under the action of minimal
  parabolic subgroups.
\newblock {\em J. Math. Soc. Japan}, 31(2):331--357, 1979.

\bibitem{PS:89}
U.~Pinkall and I.~Sterling.
\newblock On the classification of constant mean curvature tori.
\newblock {\em Annals of Math.}, 130(2):407--451, 1989.

\bibitem{Pohlmeyer:76}
K.~Pohlmeyer.
\newblock Integrable {H}amiltonian systems and interactions through quadratic
  constraints.
\newblock {\em Comm. Math. Phys.}, 46:207--221, 1976.

\bibitem{Rawnsley:84}
John Rawnsley.
\newblock Noether's theorem for harmonic maps.
\newblock In {\em Differential geometric methods in mathematical physics
  ({J}erusalem, 1982)}, volume~6 of {\em Math. Phys. Stud.}, pages 197--202.
  Reidel, Dordrecht, 1984.

\bibitem{Uhlenbeck:89}
K.~Uhlenbeck.
\newblock Harmonic maps into {L}ie groups: classical solutions of the chiral
  model.
\newblock {\em J. Differential Geom.}, 30(1):1--50, 1989.

\bibitem{Vavilov:01}
N.~A. Vavilov.
\newblock Do it yourself structure constants for {L}ie algebras of types
  {$E_l$}.
\newblock {\em Zap. Nauchn. Sem. S.-Peterburg. Otdel. Mat. Inst. Steklov.
  (POMI)}, 281(Vopr. Teor. Predst. Algebr. i Grupp. 8):60--104, 281, 2001.

\bibitem{Vogan:83}
David~A. Vogan.
\newblock Irreducible characters of semisimple {L}ie groups. {III}. {P}roof of
  {K}azhdan-{L}usztig conjecture in the integral case.
\newblock {\em Invent. Math.}, 71(2):381--417, 1983.

\bibitem{Wood:94}
J.C. Wood.
\newblock Harmonic maps into symmetric spaces and integrable systems.
\newblock In {Fordy, A.P.} and {Wood, J.C.}, editors, {\em Harmonic Maps and
  Integrable Systems}, volume E23 of {\em Aspects Math.}, pages 29--55. Vieweg,
  Braunschweig, 1994.

\end{thebibliography}

\end {document}